\documentclass[a4paper,11pt]{article}
\usepackage[utf8x]{inputenc}

\usepackage{a4wide}
\usepackage{amssymb}
\usepackage{amsmath}
\usepackage{amsthm}
\usepackage[mathscr]{euscript}
\usepackage[normalem]{ulem}
\usepackage[colorlinks=true]{hyperref}
\usepackage[usenames,dvipsnames]{xcolor}
\usepackage{appendix}
\usepackage{tikz}
\usetikzlibrary{arrows,angles}
\usetikzlibrary{matrix}
\usetikzlibrary{decorations}
\usetikzlibrary{arrows,calc,shapes,decorations.pathreplacing}
\usepackage{tikz-cd}
\usepackage{todonotes}
\usepackage{verbatim}

\numberwithin{equation}{section}
\theoremstyle{definition}

\newtheorem{Definition}{Definition}[section]

\newtheorem{Remark}[Definition]{Remark}
\newtheorem{Example}[Definition]{Example}

\theoremstyle{plain}

\newtheorem{Theorem}[Definition]{Theorem}
\newtheorem{Proposition}[Definition]{Proposition}
\newtheorem{Lemma}[Definition]{Lemma}
\newtheorem{Corollary}[Definition]{Corollary}

\newcommand{\R}{\mathbb R}

\newcommand{\Fl}{\text{\rm Fl}}

\usepackage[xcolor]{changebar}
\cbcolor{blue}
\DeclareMathOperator{\Ric}{Ric}

\newcommand{\pt}{\partial}

\newcommand{\cM}{\mathcal M}
\newcommand{\cR}{\mathcal R}

\usepackage[inline]{enumitem}
\newcommand{\enumlabelformat}{\roman}
\newcommand{\enumlabelfont}[1]{#1}
\newlength{\thelabelsep}
\setlist{labelsep=\thelabelsep}
\setlist[enumerate,1]{font=\enumlabelfont,label=(\enumlabelformat*),leftmargin=2.5em}
\setlist[itemize]{leftmargin=2.5em,label=$-$}

\newcounter{inlineenum}
\renewcommand{\theinlineenum}{\enumlabelformat{inlineenum}}

\let\epsilon\varepsilon
\let\phi\varphi

\makeatletter
\let\save@mathaccent\mathaccent
\newcommand*\if@single[3]{%
  \setbox0\hbox{${\mathaccent"0362{#1}}^H$}%
  \setbox2\hbox{${\mathaccent"0362{\kern0pt#1}}^H$}%
  \ifdim\ht0=\ht2 #3\else #2\fi
  }
\newcommand*\rel@kern[1]{\kern#1\dimexpr\macc@kerna}
\newcommand*\widebar[1]{\@ifnextchar^{{\wide@bar{#1}{0}}}{\wide@bar{#1}{1}}}
\newcommand*\wide@bar[2]{\if@single{#1}{\wide@bar@{#1}{#2}{1}}{\wide@bar@{#1}{#2}{2}}}
\newcommand*\wide@bar@[3]{%
  \begingroup
  \def\mathaccent##1##2{%
    \let\mathaccent\save@mathaccent
    \if#32 \let\macc@nucleus\first@char \fi
    \setbox\z@\hbox{$\macc@style{\macc@nucleus}_{}$}%
    \setbox\tw@\hbox{$\macc@style{\macc@nucleus}{}_{}$}%
    \dimen@\wd\tw@
    \advance\dimen@-\wd\z@
    \divide\dimen@ 3
    \@tempdima\wd\tw@
    \advance\@tempdima-\scriptspace
    \divide\@tempdima 10
    \advance\dimen@-\@tempdima
    \ifdim\dimen@>\z@ \dimen@0pt\fi
    \rel@kern{0.6}\kern-\dimen@
    \if#31
      \overline{\rel@kern{-0.6}\kern\dimen@\macc@nucleus\rel@kern{0.4}\kern\dimen@}%
      \advance\dimen@0.4\dimexpr\macc@kerna
      \let\final@kern#2%
      \ifdim\dimen@<\z@ \let\final@kern1\fi
      \if\final@kern1 \kern-\dimen@\fi
    \else
      \overline{\rel@kern{-0.6}\kern\dimen@#1}%
    \fi
  }%
  \macc@depth\@ne
  \let\math@bgroup\@empty \let\math@egroup\macc@set@skewchar
  \mathsurround\z@ \frozen@everymath{\mathgroup\macc@group\relax}%
  \macc@set@skewchar\relax
  \let\mathaccentV\macc@nested@a
  \if#31
    \macc@nested@a\relax111{#1}%
  \else
    \def\gobble@till@marker##1\endmarker{}%
    \futurelet\first@char\gobble@till@marker#1\endmarker
    \ifcat\noexpand\first@char A\else
      \def\first@char{}%
    \fi
    \macc@nested@a\relax111{\first@char}%
  \fi
  \endgroup
}
\makeatother

\DeclareMathOperator{\md}{md}
\DeclareMathOperator{\Riem}{Riem}

\title{Almost-splitting of spacetimes under timelike sectional curvature bounds}

\author{Nicola Gigli\thanks{{\tt ngigli@sissa.it}, SISSA, Trieste.}\\Melanie Graf\thanks{{\tt melanie.graf@uni-hamburg.de}, Department of Mathematics, University of Hamburg.}\\Robert J.\ McCann\thanks{{\tt mccann@math.toronto.edu}, Departments of Mathematics and of Economics, University of Toronto.}\\Argam Ohanyan\thanks{{\tt argam.ohanyan@utoronto.ca}, Department of Mathematics, University of Toronto.}\\Eric Woolgar\thanks{{\tt ewoolgar@ualberta.ca}, Department of Mathematical and Statistical Sciences, University of Alberta.}\\Matteo Zanardini\thanks{{\tt mzanardi@sissa.it}, SISSA, Trieste.}
}
\begin{document}

\date{\today}


\maketitle

\begin{abstract}
We establish an almost-splitting theorem for spacetimes under timelike sectional curvature bounds in the spirit of the celebrated Riemannian result due to Cheeger and Colding. More concretely, we show that if a globally hyperbolic spacetime with timelike sectional curvature bounded below by a nonpositive constant close to zero contains a long maximizing timelike segment, then causal diamonds along that segment are close in the Lorentzian Gromov--Hausdorff sense of Minguzzi--Suhr to those in a product metric spacetime. Our result can be interpreted as a quantitative analogue of the Lorentzian splitting theorem for spacetimes with nonnegative timelike sectional curvature by Beem--Ehrlich--Markvorsen--Galloway. In contrast to the Riemannian setting,  the timelike sectional curvature bound cannot be replaced by Ricci-flatness, as the small mass limit of the Schwarzschild exterior geometry shows.
\vspace{1em}

\medskip


\end{abstract}
\tableofcontents

\section{Introduction}
\label{section: introduction}

The celebrated results of Hawking and Penrose \cite{PenroseSingularity, HawkingSingularity, HawkingSingularityII, HawkingPenroseSingularity} assert that spacetimes are generically singular, i.e., under reasonable curvature and initial data assumptions, there exist incomplete causal geodesics.\ We refer to the standard sources (Hawking--Ellis \cite{HawkingEllis}, O'Neill \cite{bookONeill} and Beem--Ehrlich--Easley \cite{BeemBook}) for a detailed exposition of these results. 
It turns out that if an infinitely long maximizing timelike geodesic exists in a globally hyperbolic spacetime with suitably nonnegative curvature, then no such geodesic incompleteness occurs, and in fact the dynamics of the spacetime trivialize entirely. 
This summarizes (one well-known formulation of) the Lorentzian timelike splitting theorem.

The history of splitting theorems in Lorentzian geometry goes back to Yau \cite{YauProblemSection}, who asked whether a Lorentzian analogue of the Cheeger--Gromoll \cite{CheegerGromoll} splitting theorem from Riemannian geometry could be proven for timelike geodesically complete spacetimes with nonnegative timelike Ricci curvature containing a timelike line. As timelike geodesic completeness is usually not the suitable completeness notion for comparison-geometric purposes, splitting theorems have often been studied under global hyperbolicity instead. The first milestone was the work of Beem--Ehrlich--Markvorsen--Galloway \cite{BEMG84, BEMG85}, who proved a Lorentzian splitting theorem for globally hyperbolic spacetimes under nonnegative timelike sectional curvature (rather than Ricci) bounds. The improvement to nonnegative timelike Ricci bounds (i.e., the physically relevant strong energy condition) was made by Eschenburg \cite{Eschenburgsplitting}, but he had to assume both global hyperbolicity and timelike geodesic completeness. The latter assumption in Eschenburg's result was subsequently removed by Galloway \cite{Galloway89splitting}, and finally Yau's initial question (i.e., assuming timelike geodesic completeness rather than global hyperbolicity) was answered in the affirmative by Newman \cite{Newmansplitting}. (Newman's proof was later simplified by Galloway--Horta \cite{GallowayHorta95}). Versions of Galloway's and Newman's results under Bakry--Émery Ricci curvature bounds were subsequently established by Case \cite{CaseSingularitySplitting} and Woolgar--Wylie \cite{WoolgarWylieSplittingI, WoolgarWylieSplittingII}. Let us also mention the recent method of proving Lorentzian splitting theorems using the nonlinear and (nonuniformly) elliptic $p$-d'Alembertian, which first appeared in Mondino--Suhr \cite{MondinoSuhr} and was subsequently systematically investigated in a general nonsmooth spacetime setting in the work of Beran--Braun--Calisti--Gigli--McCann--Ohanyan--Rott--Sämann \cite{beran2024nonlinear}. It was first used to give a simpler proof of both Galloway's and Newman's result by Braun--Gigli--McCann--Ohanyan--Sämann \cite{QuintetI}, then generalized to ($C^1$-weighted) metrics which are $C^1$-regular by the same authors \cite{QuintetII}, and to the Lorentz--Finsler setting by Caponio--Ohanyan--Ohta \cite{CaponioOhanyanOhta}.

In Riemannian signature, the groundbreaking work of Cheeger--Colding \cite{CheegerColding96} establishes a quantitative version of the Cheeger--Gromoll theorem. More specifically, it states that if the Ricci curvature is almost nonnegative, and if there exists a long almost-minimizing curve, then sufficiently small balls around the midpoint of that curve are close in the Gromov--Hausdorff sense to balls in a product metric space (which need not be a smooth manifold). The error goes to zero as the inverse of the length, the minimization error, and the curvature bound below all go to zero. The result of Cheeger--Colding highlights the fact that non-smooth metric spaces arise even if the space being studied is a smooth Riemannian manifold, and is the starting point of the rich theory of Ricci limit spaces \cite{RicciLimit1, RicciLimit2, RicciLimit3}.

Up to now, no quantitative splitting theorem has been available for spacetimes. The purpose of this work is to establish the first such result. More specifically, we will prove a quantitative version of the splitting theorem of Beem--Ehrlich--Markvorsen--Galloway under nonnegative timelike sectional curvature\footnote{We use the signature $(+,-,\dots,-)$ rather than the more common $(-,+,\dots,+)$. Hence, nonnegative timelike sectional curvature for us corresponds to nonpositive timelike sectional curvature in \cite{BEMG84, BEMG85} and other works. We refer to Section \ref{subsection: timelikesectional} for more details.}, which we now recall:

\begin{Theorem}[Beem--Ehrlich--Markvorsen--Galloway \cite{BEMG84, BEMG85}]
Let $(\mathcal M,\mathsf g)$ be a globally hyperbolic spacetime with nonnegative timelike sectional curvature that contains a complete maximizing timelike geodesic $\Upsilon :\R \to \mathcal M$ (i.e., a complete timelike line). Then there exists a complete Riemannian manifold $(\mathcal S,\mathsf h)$ of nonnegative sectional curvature such that $(\mathcal M, \mathsf g) \cong (\R \times S, dt^2 - \mathsf h)$. In particular, $(\mathcal M,\mathsf g)$ is geodesically complete.
\end{Theorem}

The proof of this theorem uses in an essential way the Lorentzian Toponogov triangle comparison theorem by Harris \cite{harris1982triangle}, which was recently improved to yield a full equivalence between smooth and synthetic notions by Beran--Kunzinger--Ohanyan--Rott \cite{BKOR24}, building on the globalization results of Beran--Harvey--Napper--Rott \cite{BHNR}. It was also extended to Kunzinger--Sämann's setting of Lorentzian length spaces \cite{KunzingerSaemannLLS} (a Lorentzian analogue of metric length spaces) by Beran--Ohanyan--Rott--Solis \cite{BORS} and the global curvature assumptions therein were later relaxed in \cite{BHNR}.

In order to prove a quantitative version of the Beem--Ehrlich--Markvorsen--Galloway theorem, we make use of the Lorentzian Gromov--Hausdorff distance introduced and studied by Minguzzi--Suhr \cite{MinguzziSuhr22} in the context of bounded Lorentzian metric spaces (see also their recent extension to the unbounded setting with Bykov \cite{BykovMinguzziSuhr}, cf. Braun--McCann \cite{BraunMcCann23+}, as well as an alternative approach to Lorentzian Gromov--Hausdorff convergence by Mondino--Sämann \cite{MondinoSaemann2025}). Let us now present our main result. Detailed definitions of the objects which appear can be found in Section \ref{section: preliminary}.

\begin{Theorem}[Almost-splitting under timelike sectional curvature bounds]
\label{Theorem: main result intro}
Fix $R > 0$. 
For every $\delta \ge 0$ and $L \gg R$,
there exists a quantity $\Psi = \Psi(\delta,L \, | R)$ with 
$\Psi(\delta,L \, | R) \to 0$ as $ (\delta,L^{-1}) \to 0$ 
such that the following holds: If $(\mathcal M,\mathsf g)$ is a globally hyperbolic spacetime with timelike sectional curvature bounded below by $-\delta \leq 0$ and $\Upsilon:[-L,L] \to \mathcal M$ is a maximizing timelike geodesic parametrized by Lorentzian arclength, then there exists a pointed metric space $(X,\mathsf d, x_0)$ (depending on $\delta,R,L,\Upsilon$ and $(\cM,\mathsf g)$) such that
\begin{equation}
    \mathsf d_{LGH}(J^\mathcal M(\Upsilon_{-R},\Upsilon_R), J^{\R \times X}((-R,x_0),(R,x_0))) \leq \Psi(\delta,L\,|\,R)
\end{equation}
where $\mathsf d_{{LGH}}$ and the time-separation $\ell$ on $\R \times X$ are from Definitions 
\ref{D:d_LGH}--\ref{Definition: Lorentzian products} below.
\end{Theorem}

\begin{Remark}[On the structure of the factors]
We shall eventually construct $X$ as the intersection of the diamond $J^{\mathcal M}(\Upsilon_{-R},\Upsilon_R)$ with the level set of $\ell(\Upsilon_{-L},\cdot)$ 
from \eqref{D:ell} passing through $x_0 := \Upsilon_0$.  The monotonicity and semiconvexity \cite[Prop. 3.4]{McCann20} of 
$\ell(\Upsilon_{-L},\cdot)$ imply this level set is a locally Lipschitz (in fact DC \cite[Cor. 2.6]{KitagawaMcCann19}) 
spacelike hypersurface. Moreover, outside of the spacelike boundary $i^0$ of $J^{\mathcal M}(\Upsilon_{-R},\Upsilon_R)$ identified in Remark \ref{Remark: i^0} --- we expect the $DC$ implicit function theorem
\cite{VeselyZajicek89} to imply that
$X\setminus i^0$ is a locally $DC$ hypersurface-with-boundary on which the distance $\mathsf d$ we construct metrizes the discrete topology. Here $DC$ refers to manifolds in which the transition functions between charts are coordinatewise differences of convex functions; in particular,  such submanifolds are $C^2$-rectifiable.
\end{Remark}

The article is organized as follows. Subsection \ref{subsection: Notationandconventions} fixes our notation and conventions throughout the paper. 

In Section \ref{section: preliminary}, we discuss preliminary notions and results that will be used throughout the proof of our main result: In Subsection \ref{subsection: timelikesectional}, we recall the notion of timelike sectional curvature bounds on spacetimes and its associated characterizations via geometric comparison in the context of the Lorentzian Toponogov theorem --- a suitable convexity statement concerning the $-\delta$-modified Lorentzian distance $\md^{-\delta}$ in particular. In Subsection \ref{subsection: boundedLorentzianmetricspaces and GH-convergence}, we introduce the synthetic setting of metric spacetimes and the Lorentzian Gromov--Hausdorff distance insofar as we will need it, following Minguzzi--Suhr \cite{MinguzziSuhr22}. In Subsection \ref{subsection: causally convex functions grad flows}, we discuss causal and causally convex functions, in particular properties of the gradient flows of the latter {as recently obtained} by Braun--Gigli--McCann--Vincini--Zanardini \cite{BraunGigliMcCannVinciniZanardini}. These results are tied to our context via the observation that the modified Lorentzian distance functions to and from the endpoints of a long maximizing segment are quantifiably causally convex, cf.\ Theorem \ref{Theorem: convexityoftimeseparation}. 

The main part of our article is Section \ref{Section: Almost splitting}. We recall our main result in Subsection \ref{subsection: main result}, to whose proof the subsequent subsections are dedicated. In Subsection \ref{Subsection: expansion estimates, monotonicity, reparam flows}, we study expansion rate estimates for gradient flows of causally ($\lambda$-)convex functions and resulting monotonicity properties. Crucially, we observe that the causal convexity of the modified Lorentzian distance does not in general pass to causal convexity of the Lorentzian distance itself, thus we use reparametrization to define the gradient flow of the latter. Our timelike sectional curvature bound gives us a way to quantifiably control this reparametrization, obtaining expansion estimates and monotonicity formulae for the reparametrized flow as well. In Subsection \ref{Subsection: Busemann sum estimates, almost linearity, almost Pythagoras}, as a consequence of our work in the previous subsection, we give quantitative estimates on the sum of the two Lorentzian distances associated with a long timelike segment (equivalently, the sum of the approximate Busemann functions), and we also obtain almost-rigidity statements on the Lorentzian distances along their flow in the form of almost-linearity. Various formulae for almost product geometries follow, which ultimately allow us to metrize the intersection of a Busemann level set with a diamond having fixed vertices on long segment by using the flows. Having these results at hand, we are able to conclude the proof of our main theorem in Subsection \ref{Subsection: Proof of main result}, by constructing a correspondence of small distortion between a fixed causal diamond with vertices on our long maximizing segment and the corresponding diamond in a product metric spacetime. 

In Section \ref{section: discussion outlook}, we give an outlook on possible future related lines of research.

We also include Appendix \ref{section: counterexample}, where we give a counterexample to a version of Theorem \ref{Theorem: main result intro} where the timelike sectional curvature bound is replaced by a timelike Ricci bound, i.e., to a direct Lorentzian analogue of the Cheeger--Colding theorem. Explicitly, we show that one can pick longer and longer timelike maximizers in exterior Schwarzschild spacetime of mass $\mu$, such that fixed diamonds around their center do not approach product geometry as $\mu \downarrow 0$.

\subsection{Notation \& conventions}
\label{subsection: Notationandconventions}

Lorentzian metrics $\mathsf g$ are smooth and have signature $(+,-,\dots,-)$.
Vectors $v$ with $\mathsf g(v,v) > 0$ are called timelike and those with $\mathsf g(v,v) < 0$ are called spacelike. Vectors $v$ obeying $\mathsf g(v,v)=0$ are called null. A causal vector is by definition one that is either timelike or null. 
A Lorentzian manifold $(\mathcal M,\mathsf g)$ (i.e., $\mathcal M$ a smooth manifold and $\mathsf g$ a Lorentzian metric) is time-orientable if it admits a global smooth timelike vector field $V$, a time-orientable Lorentzian manifold is called \emph{spacetime}. Given a vector $v$ and a co-vector $\omega$, we write $v^\flat$ and $\omega^\sharp$ for the co-vector resp.\ vector corresponding to $v$ resp.\ $\omega$ via the musical isomorphisms induced by $\mathsf g$. We denote by $\leq$ and $\ll$ the causal and chronological orders, respectively. A spacetime is globally hyperbolic if it contains no closed causal curves and causal ``diamonds'' $J(x,y):=J^+(x)\cap J^-(y)$ are compact for all $x,y \in \mathcal M$. A causal vector $v \in T_p\mathcal M$ is future directed (sometimes simply future) if $\mathsf g_p(v,V(p)) \geq 0$, past directed (or simply past) if $\mathsf g_p(v,V(p)) \leq 0$. We sometimes write $\|v\|_{\mathsf g}:=\sqrt{\mathsf g(v,v)}$ if $v$ is future directed causal.

Spacetimes have dimension $n$ throughout this article. The Riemann curvature tensor $\Riem$ is defined by 
\begin{align}
    \Riem(U,V)W:=\nabla_U \nabla_V W - \nabla_V \nabla_U W - \nabla_{[U,V]} W
\end{align}
for vector fields $U,V$ and $W$ on $\cM$. 

We will employ the notation $\Psi= \Psi(\delta , L \, | \, R)$ for the error in our quantitative splitting. This function depends on the variables $\delta, L,R$ and has the important property that, if we evaluate the function at some given fixed $R=R'$, it converges to $0$ if we let $\delta \downarrow 0$ and $L \uparrow + \infty$. A similar convention is used in \cite{CheegerColding96}. Let us also note that the exact value of $\Psi$ may change from calculation to calculation, but its characteristic convergence property does not change. We will sometimes write $(\delta, L^{-1}) \to 0$ to mean $\delta \to 0$ and $L \to \infty$. Since we will usually fix the value of $R$ throughout our proof, we will often omit it in the notation, i.e., we will write $\Psi(\delta,L)$ instead of $\Psi(\delta,L \, | \, R)$. We emphasize that $\Psi$ is sometimes manifestly nonnegative and can sometimes take both signs, it is clear from the context whether one type of error or the other is meant.

On a spacetime $(\mathcal M,\mathsf g)$ a function $f: \mathcal M \rightarrow \R$ is called locally Lipschitz if it is locally Lipschitz with respect to any Riemannian distance on $M$. It is easily verified that any locally Lipschitz function in this sense is Lipschitz when restricted to a compact subset of $M$ with respect to any given Riemannian distance. 

Given a curve $\sigma:[a,b] \to M$, we interchangeably write the evaluation at $t$ either as $\sigma(t)$ or as $\sigma_t$. We use the term {\em Lorentzian arclength} of a causal curve (denoted $L_{\mathsf g}(\gamma)$) synonymously with {\em proper time}.  We prefer to use {\em time-separation} for the function $\ell:\cM^2 \longrightarrow \{-\infty\} \cup [0,\infty]$ of two variables that encodes the geometry, defined via
\begin{equation}\label{D:ell}
    \ell(x,y):=\sup\{L_{\mathsf g}(\gamma): \gamma \text{ future causal from }x \text{ to } y\},
\end{equation}
but 
{\em Lorentzian distance} for the {time-separation} $\ell(\cdot,z)$ (or $\ell(x,\cdot)$) to a fixed event $z$ in the future (or $x$ in the past); we occasionally use the two terms interchangeably, however.  Similarly, we prefer to reserve $\cM$ for a smooth spacetime, whereas $X$, $Y$ and $Z$ denote metric spaces or metric spacetimes that need not be smooth. We denote by $\ell_+:=\max(\ell,0)$ the nonnegative part of $\ell$.

Given a real-valued function $f$, we denote its right and left Dini derivatives at $t = t_0$ by
\begin{equation*}
    \frac{d}{dt}\bigg|^+_{t = t_0} f(t) :=\liminf_{h \to 0^+} \frac{f(t_0 + h) - f(t_0)}{h}, \quad \frac{d}{dt}\bigg|^-_{t = t_0} f(t):=\limsup_{h \to 0^+} \frac{f(t_0) - f(t_0-h)}{h}.
\end{equation*}
\section{Preliminary material}
\label{section: preliminary}

\subsection{Timelike sectional curvature bounds}
\label{subsection: timelikesectional}

Let us start our discussion by recalling the classical notion of timelike sectional curvature bounds on smooth spacetimes $(\mathcal M,\mathsf g)$.

If $\{v,w\}$ spans a non-degenerate plane $\Pi$ in a tangent space $T_p\mathcal M$, then the \textit{sectional curvature} of $\Pi$ is defined by 
\begin{equation}
    K(\Pi) := \frac{\mathsf g(\Riem(w,v)v,w)}{\mathsf g(v,v) \mathsf g(w,w) - \mathsf g(v,w)^2}.
\end{equation}
It is easily seen that this definition is independent of the choice of basis $\{v,w\}$ of $\Pi$.

\begin{Definition}[Timelike sectional curvature bounds] \label{def: timelike sectional curvature bound}
A spacetime $(\mathcal M,\mathsf g)$ is said to have a \textit{timelike sectional curvature bound} from below (resp.\ above) by a constant $k \in \R$ if for all $p \in \mathcal M$ and all timelike planes $\Pi \subseteq T_p\mathcal M$, we have
\begin{equation}
    K(\Pi) \geq k \quad (\text{resp.\ } \leq).
\end{equation}
\end{Definition}

\begin{Remark}[On timelike sectional curvature bounds]
\begin{enumerate} \item[] 
    \item If $\Pi$ is a timelike plane spanned by orthonormal vectors $v,w$ with $v$ unit timelike (i.e., $\mathsf g(v,v) = +1$) and $w$ unit spacelike (i.e., $\mathsf g(w,w) = -1$), then the sectional curvature of $\Pi$ reduces to
    \begin{align}
        K(\Pi) = -\mathsf g(\Riem(w,v)v,w).
    \end{align}
    \item If the timelike sectional curvature of $(\mathcal M,\mathsf g)$ is bounded from below (resp.\ above) by $k \in \R$, then its timelike Ricci curvature is bounded from below (resp.\ above) by $(n-1)k\mathsf g$: Indeed, suppose for simplicity that the bound is from below by $k$ and that $v \in T_p \mathcal M$ is unit timelike, and complete $\{ v \}$ to a basis of $T_p \mathcal M$ by appending unit spacelike vectors $e_2,\dots,e_n$ ($e_1:=v$). Noting that $-\mathsf g(\Riem(e_i,v)v,e_i) \geq k$ for every $i=2,\dots,n$, we get
    \begin{equation}
        \Ric(v,v) = \sum_{i=2}^n g(e_i,e_i) \mathsf g(\Riem(e_i,v)v,e_i) = -\sum_{i=2}^n \mathsf g(\Riem(e_i,v)v,e_i) \geq (n-1)k.
    \end{equation}
    \item Note that, due to our choice of signature $(+,-,\dots,-)$ of the Lorentzian metric $\mathsf g$, it is easy to see that our definition of sectional curvature is precisely the negative of the ones often considered in Lorentzian geometry, where the signature is $(-,+,\dots,+)$, e.g. \cite{BKOR24}. Hence, note that timelike sectional curvature bounds from below (resp.\ above) by $k \in \R$ in our sense correspond to bounds by $-k$ from above (resp.\ below) in the sense of \cite{BKOR24}. While in principle the role of the constant curvature spaces also switches for us (i.e., anti-de Sitter space has positive timelike sectional curvature and de-Sitter space negative), the comparison of various geometric configurations with constant curvature spaces remains the same as outlined in \cite{BKOR24} keeping in mind the passage of the bounds from our conventions to those of \cite{BKOR24}, as discussed above.

    \item Definition \ref{def: timelike sectional curvature bound} is the correct notion of sectional curvature bound in the Lorentzian setting. Indeed, if one were to require the same curvature bound for all $2$-planes in the tangent spaces, 
    as is done in the Riemannian setting, it is well known that the spacetime then would have constant sectional curvature if $n \geq 3$, see Kulkarni \cite{Kulkarni}. 
\end{enumerate}
\end{Remark}

Just like in the Riemannian case, in the spacetime setting it is known that timelike sectional curvature bounds imply bounds on various geometric configurations when compared to ones in the corresponding constant curvature ($2$-dimensional) spacetime. The first observation in this direction is due to Harris \cite{harris1982triangle}, whereas the equivalence of all of these notions in strongly causal smooth spacetimes is due to Beran--Kunzinger--Ohanyan--Rott \cite{BKOR24} (in conjunction with the equivalence of the geometric conditions in Lorentzian length spaces by Beran--Kunzinger--Rott \cite{BKR}).

We will discuss these results insofar as they pertain to our proof of the almost splitting. Let us denote by $\mathbb L^2(k)$ the simply connected $2$-dimensional spacetime of constant sectional curvature $k$. As we already mentioned, our choice of signature together with our definition of timelike sectional curvature bounds ensure that if $k > 0$ then $\mathbb L^2(k)$ is an anti-de Sitter space and if $k < 0$ then $\mathbb L^2(k)$ is a de-Sitter space. We write $\bar{\ell}_k$ for the time separation of $\mathbb L^2(k)$.

Let us now formulate the Lorentzian Toponogov theorem (see for example the classical result \cite[Thm.\ 3.4]{harris1982triangle} or the more recent version \cite[Thm.\ 3.2]{BKOR24}). Since we only require this theorem for the case of curvature bounded below by $-\delta \leq 0$, we will omit a discussion of triangle size bounds. Specifically, we recall the version \cite[Thm.\  3.4]{harris1982triangle} of Harris which, in contrast to modern treatments of the Lorentzian Toponogov theorem, one side is allowed to be merely locally maximizing and timelike. For future use, we recall that in $\mathbb L^2(k)$ with $k \leq 0$, any two timelike related points are connected by a unique (hence maximizing) timelike geodesic, cf.\ \cite[Prop.\ 11.13]{BeemBook}. 

\begin{Theorem}[Timelike triangle comparison] \label{thm: triangle comparison}
Let $(\mathcal M,\mathsf g)$ be a globally hyperbolic spacetime with timelike sectional curvature bounded below by $- \delta \leq 0$. Given any timelike triangle $\Delta(p_1,p_2,p_3)$ in $M$, i.e.\ any choice of timelike related points $p_1 \ll p_2 \ll p_3$ as well as a choice of connecting timelike maximizers from $p_1$ to $p_3$ and from $p_1$ to $p_2$, and a not necessarily maximizing timelike geodesic $\sigma:[0,1] \to \mathcal M$ from $p_2$ to $p_3$, let $\Delta(p_1^*,p_2^*,p_3^*)$ be the (up to isometries unique) comparison triangle in $\mathbb L^2(-\delta)$, i.e. $p_1^* \ll p_2^* \ll p_3^*$ and $\ell(p_1,p_2) = \bar{\ell}_{-\delta}(p_1^*,p_2^*)$, $\ell(p_1,p_3) = \bar{\ell}_{-\delta}(p_1^*,p_3^*)$ and $L_\mathsf g(\sigma) = \bar{\ell}_{-\delta}(p_2^*,p_3^*)$. Given any point $q$ on the side $p_2p_3$ (i.e., on $\sigma$) and the corresponding point $q^*$ on the side of $p_2^{\ast}p_3^{\ast}$ (i.e., if $q = \sigma_t$, the unique $t$-midpoint on the side $p_2^*p_3^*$), we have
\begin{equation} \label{eq: triangle comparison}
    \ell(p_1,q) \leq \bar{\ell}_{-\delta}(p_1^*,q^*).
\end{equation}
\end{Theorem}

The following (equivalent) characterization of timelike sectional lower bounds via convexity properties of the modified time separation function will prove useful, cf.\ \cite[Thm.\ 4.37]{BKR}. We note that the proof in that reference also yields the convexity along not necessarily maximizing geodesics upon use of Harris's version of the timelike triangle comparison that we have cited above.

\begin{Theorem}[Lorentzian Toponogov theorem -- convexity version]
\label{Theorem: Toponogov convexity version}
Let $(\mathcal M,\mathsf g)$ be a globally hyperbolic spacetime with timelike sectional curvature bounded below by $-\delta \leq 0$. Then for any $p \in M$ and any (not necessarily maximizing) timelike geodesic $\gamma:[0,1] \to \mathcal M$ contained in $I^+(p)$, the function $f_\delta:[0,1] \to \R$ defined by
\begin{equation}
    f_\delta(t):=\begin{cases}
        \frac{\ell(p,\gamma_t)^2}{2}, \quad \delta = 0,\\
        \frac{\cosh(\sqrt{\delta} \, \ell(p,\gamma_t)) - 1}{\delta}, \quad {\delta > 0},
    \end{cases}
\end{equation}
satisfies the following differential inequality in the distributional sense:
\begin{equation}
    f_\delta''(t) \geq L_\mathsf g(\gamma)^2 (1 {+} \delta f_\delta(t)). 
\end{equation}
\end{Theorem}

We now turn to a discussion of the Lorentzian (hyperbolic) law of cosines in simply connected $2$-dimensional spacetimes of constant curvature $\mathbb L^2(-\delta)$ (with $\delta \geq 0$) which will be very useful throughout, cf.\ Beran--Sämann \cite[Lem.\ 2.8]{Beran_2023} (notice that this reference adopts the signature convention $(-, +, \dots, +)$ that flips the sign of the sectional curvature on timelike planes). We will use all of the variants of this result listed below repeatedly in our proof of the almost-splitting.

\begin{Proposition} \label{Thm : Law of cosines} \emph{(Hyperbolic law
of cosines \cite[Lem.\ 2.8]{Beran_2023})} 
\begin{enumerate}[label = (\Alph*)]
    \item Let $\delta \geq 0$ and $\mathbb L^2(-\delta)$ be the simply connected $2$-dimensional spacetime with timelike sectional curvature equal to $-\delta$ and time separation function given by $\ell = \overline{\ell}_{-\delta}$. Consider a timelike triangle $\Delta (p_1,p_2,p_3)$ in $\mathbb L^2(-\delta)$ with sidelengths
    \[
    \ell (p_1,p_2) =a, \qquad \ell(p_2, p_3) = b, \qquad \ell(p_1,p_3)=c+d,
    \]
    and suppose that $p_1p_3$ is the longest side, namely $a+b \leq c+d$. Choose now a point $q$ on the maximizing timelike geodesic $\sigma$ connecting $p_1$ to $p_3$ such that
    \[
    \ell (p_1, q)= c, \qquad \ell (q, p_3) = d. 
    \]
    We suppose $q \leq p_2$ or $p_2 \leq q$ and define $x:= \max \{ \ell (q, p_2), \ell(p_2,q) \}$. Then the following formulas hold:
    \begin{itemize}
        \item[i)] If $\delta =0$, then 
        \begin{equation*}
        x^2 = \frac{c b^2+a^2 d}{c+d} - cd.
        \end{equation*}
        \item[ii)] If $\delta >0$, then 
        \begin{equation*}
        \begin{split}
    \cosh( \sqrt{\delta}x ) =&\, \left ( \cosh( \sqrt{\delta}b ) - \cosh( \sqrt{\delta}a) \cosh( \sqrt{\delta}(c+d))  \right ) \cdot  
    \dfrac{\sinh (\sqrt{\delta}c )}{\sinh (\sqrt{\delta}(c+d)) }\\
    &\, + \cosh ( \sqrt{\delta} a) \cosh ( \sqrt{\delta} c).
    \end{split}
    \end{equation*}
    In particular, the following asymptotic formula holds:
    \begin{equation*}
    x^2 = \dfrac{cb^2+a^2d}{c+d} - cd + O(\delta (c+d)^4) \text{ for } \delta \downarrow 0 \text{ or } (c+d) \downarrow 0.
    \end{equation*}
    \end{itemize}

    \item Similarly to $(A)$, consider a timelike triangle $\Delta(p_1,p_2,p_3)$ in $\mathbb L^2(-\delta)$ with sidelengths $a = \ell(p_1,p_2)$, $b+c = \ell(p_2,p_3)$ and $d=\ell(p_1,p_3)$, such that $a+b+c \leq d$. Let $q$ be a point on the side from $p_2$ to $p_3$ such that $\ell(p_2,q) = b$ and $\ell(q,p_3) = c$. Then for $x:=\ell(p_1,q)$ the following formulas hold:
    \begin{itemize}
        \item[i)] If $\delta = 0$, then
        \begin{equation*}
            x^2 = \frac{bd^2 + a^2 c}{b+c}-bc.
        \end{equation*}
        \item[ii)] If $\delta > 0$, then
        \begin{align*}
            \cosh(\sqrt{\delta} x) = &(\cosh(\sqrt{\delta d}) - \cosh(\sqrt{\delta}a)\cosh(\sqrt{\delta}(b+c))) \cdot \frac{\sinh(\sqrt{\delta} b)}{\sinh(\sqrt{\delta}(b+c))}\\ 
            &+ \cosh(\sqrt{\delta}a) \cosh(\sqrt{\delta}b).
        \end{align*}
        In particular, the following asymptotic formula holds:
        \begin{equation*}
            x^2 = \frac{bd^2 + a^2 c}{b+c}-bc + O(\delta d^4) \text{ for } \delta \downarrow 0 \text{ or } d \downarrow 0.
        \end{equation*}
    \end{itemize}
    \item In analogy with the previous cases, consider a timelike triangle $\Delta(p_1,p_2,p_3)$ in $\mathbb L^2(-\delta)$ with sidelengths $\ell(p_1,p_2) = a+b$, $\ell(p_2,p_3) = c$, $\ell(p_1,p_3) = d$. Let $q$ be the point on the side from $p_1$ to $p_2$ such that $\ell(p_1,q) = a$, $\ell(q,p_2) = b$. Denote by $x:=\ell(q,p_3)$. Then
    \begin{itemize}
    \item[i)] If $\delta = 0$, then
        \begin{equation*}
            x^2 = \frac{bd^2 + c^2 a}{b+a}-ba.
        \end{equation*}
        \item[ii)] If $\delta > 0$, then
        \begin{align*}
            \cosh(\sqrt{\delta} x) = &(\cosh(\sqrt{\delta d}) - \cosh(\sqrt{\delta}c)\cosh(\sqrt{\delta}(b+a))) \cdot \frac{\sinh(\sqrt{\delta} b)}{\sinh(\sqrt{\delta}(b+a))}\\ 
            &+ \cosh(\sqrt{\delta}c) \cosh(\sqrt{\delta}b).
        \end{align*}
        In particular, the following asymptotic formula holds:
        \begin{equation*}
            x^2 = \frac{bd^2 + c^2 a}{b+a}-ba + O(\delta d^4) \text{ for } \delta \downarrow 0 \text{ or } d \downarrow 0.
        \end{equation*}
    \end{itemize}
    \end{enumerate}
\end{Proposition}

\subsection{Metric spacetimes and Lorentzian Gromov--Hausdorff distance}
\label{subsection: boundedLorentzianmetricspaces and GH-convergence}

In order to compare subsets of two given spaces (as will be necessary in the almost-splitting), we will employ the Lorentzian Gromov--Hausdorff distance, as introduced by Minguzzi--Suhr in \cite{MinguzziSuhr22}. We note that that work studies bounded Lorentzian metric spaces, which are too restrictive for our purposes because of the compactness built into their definition. Because of this, we will now introduce all of the synthetic Lorentzian geometry we will need during the course of our proof of the almost-splitting.

\begin{Definition}[Metric spacetimes]
A \emph{metric spacetime} is a pair $(X,\ell)$, where $X$ is a set and $\ell:X^2 \to \{-\infty\} \cup [0,\infty]$ is a function (called \emph{time separation}) satisfying
\begin{equation}
    \ell(x,z) \geq \ell(x,y) + \ell(y,z)
\end{equation}
for all $x,y,z \in X$. Here, the right-hand side is understood to be $-\infty$ if one of the summands is $-\infty$. We say $(X,\ell)$ is \emph{bounded} if $\ell$ is bounded above.
\end{Definition}

The timelike and causal relations $\ll$ and $\leq$ on a metric spacetime $(X,\ell)$ are defined by $\ll:=\{\ell > 0\}$ and $\leq:=\{\ell \geq 0\}$, respectively. We use the familiar notation $I^\pm$, $J^\pm$ and $I(\cdot,\cdot)$, $J(\cdot,\cdot)$ for timelike/casual futures/pasts and timelike/causal diamonds from usual Lorentzian geometry.

\begin{Example}[Examples of metric spacetimes]
\begin{enumerate}
    \item[] 
    \item Any smooth  
    spacetime $(\mathcal M,\mathsf g)$ naturally has the structure of a metric spacetime $(\mathcal M,\ell)$, where $\ell$ is the usual time separation function. 
    \item Given a metric spacetime $(X,\ell)$ and $A \subseteq X$, then $(A,\ell|_{A^2})$ is also a metric spacetime.
    \item Compact subsets of globally hyperbolic spacetimes are bounded metric spacetimes.
\end{enumerate}
\end{Example}

Next, we discuss the Lorentzian Gromov--Hausdorff (semi-)distance, cf.\ \cite[Def.\ 4.6]{MinguzziSuhr22}, in the context of bounded metric spacetimes. To this end, recall that a \emph{correspondence} between sets $X$ and $Y$ is a left- and right-total relation $ \cR  \subseteq X \times Y$, i.e.\ for each $x \in X$ there is $y \in Y$ such that $(x,y) \in \cR$, and similarly for each $y \in Y$ there is an $x \in X$ such that $(x,y) \in \cR$. Clearly, $X \times Y$ is itself a correspondence.

\begin{Definition}[Lorentzian Gromov--Hausdorff distance]
\label{D:d_LGH}
Let $(X,\ell^X)$ and $(Y,\ell^Y)$ be bounded metric spacetimes.
\begin{enumerate}
\item Given a correspondence $\cR \subseteq X \times Y$, its \emph{distortion} is
\begin{equation}
    \mathrm{dis}(\cR):= \sup \{ |\ell^X_+(x,x') - \ell^Y_+(y,y')| : (x,y),(x',y') \in \cR\}.
\end{equation}
\item The \emph{Lorentzian Gromov--Hausdorff distance} $\mathsf d_{LGH}(X,Y)$ is defined as
\begin{equation}
\label{d_{LGH}}
    \mathsf d_{LGH}(X,Y) := \inf_{\cR} \mathrm{dis}(\cR), 
\end{equation}
where the infimum runs over all correspondences $\cR \subseteq X \times Y$.
\end{enumerate}
\end{Definition}

Next, we introduce product metric spacetimes, with respect to which we will formulate closeness in the context of almost-splitting.

\begin{Definition}[Lorentzian products]
\label{Definition: Lorentzian products}
Given a metric space $(Z, \mathsf d)$, consider the product $\R \times Z$. Define the time separation function $\ell_{\R \times Z}:(\R \times Z)^2 \to \{-\infty\} \cup [0,\infty)$ on $\R \times Z$ via the formula
\begin{equation}\label{Lorentzian products}
    \ell_{\R \times Z}((s,x),(t,y)):=\begin{cases}
        \sqrt{(t-s)^2 - \mathsf d(x,y)^2}, & t-s \geq \mathsf d(x,y),\\ -\infty & \text{otherwise}.
    \end{cases}
\end{equation}
We call $(\R \times Z, \ell_{\R \times Z})$ the Lorentzian product over $(Z,\mathsf d)$.
\end{Definition}

The following is elementary to prove.

\begin{Lemma}[On Lorentzian products]
\label{Lemma: On Lorentzian products}
Given any metric space $(Z,\mathsf d)$, the Lorentzian product $(\R \times Z, \ell_{\R \times Z})$ is a metric spacetime. Moreover, on any causal diamond of the form $J((-R,x_0),(R,x_0)) \subseteq \R \times Z$ for any $x_0 \in Z$ and $R \geq 0$, the time separation is bounded above by $2R$. In particular, any such causal diamond is a bounded metric spacetime when endowed with the restriction of the 
time separation.
\end{Lemma}

\begin{Remark}[On the spacelike boundary]
\label{Remark: i^0}
If $J(x,y)$ is a causal diamond in a globally hyperbolic spacetime $(\mathcal M,\mathsf g)$ (or in a product metric spacetime $(\R \times Z, \ell_{\R \times Z})$), we denote by $i^0$ (or $\bar i^0$) the set of points $w \in J(x,y)$ such that $\ell(p,w)  \leq 0$ and $\ell(w,p) \leq 0$ for all $p \in J(x,y)$. 
When we want to compute the Lorentzian Gromov--Hausdorff distance, we will first collapse the set $i^0$ and $\bar i^0$ (if non-empty) to a point, which is exactly the setting considered by Minguzzi--Suhr \cite{MinguzziSuhr22}. Since these quotients are bounded metric spacetimes with respect to the restriction of $\ell$, and the positive part of this restriction agrees with the positive part of $\ell$ on $J(x,y)$, for the purposes of computing the Lorentzian Gromov--Hausdorff distances between two such causal diamonds, one may choose to collapse $i^0$ and $\bar i^0$ separately without changing the result.
\end{Remark}

\subsection{Causal and causally convex functions, gradient flows}
\label{subsection: causally convex functions grad flows}

Let us now discuss causally $\lambda$-concave and $\lambda$-convex functions on spacetimes. Their analysis will be crucial in the proof of the almost-splitting theorem, primarily because a timelike sectional curvature lower bound implies a certain $\lambda$-convexity for modified Lorentz distance functions (i.e. time separations to and from points, Theorem \ref{Theorem: Toponogov convexity version}).

First, let us define (anti-)causal functions as well as the property of $c$-steepness. We restrict ourselves to globally hyperbolic smooth spacetimes and real-valued functions since we do not need more, but see \cite{beran2024nonlinear} for a much deeper study of (even infinite-valued) causal functions in a very general non-smooth setting.

\begin{Definition}[(Anti-)causal functions, $c$-steepness]
Let $(\mathcal M,\mathsf g)$ be a globally hyperbolic spacetime and $C \subseteq \mathcal M$.
\begin{enumerate}
    \item[(i)] A function $f:C \subseteq \mathcal M \to \R$ is called \textit{causal} if $f(x) \leq f(y)$ whenever $x,y \in C$ with $x\leq y$. It is called \textit{anti-causal} if instead $f(y) \leq f(x)$ whenever $x,y \in C$ with $x \leq y$. 
    \item[(ii)] A function $f:C \subseteq \mathcal M \to \R$ is called $c$-steep (for some number $c \geq 0$) if for every $x,y \in C$ we have that
    \begin{equation}
        f(y) - f(x) \geq c\ell(x,y).
    \end{equation}
    The function $f$ is called $c$-\textit{antisteep} if instead for all $x,y \in C$ we have that
    \begin{equation}
        f(x) - f(y) \geq c \ell(x,y).
    \end{equation}
\end{enumerate}
\end{Definition}

\begin{Definition}[Causally convex sets]
    A subset $C \subseteq \mathcal M$ of a globally hyperbolic spacetime $(\mathcal M,\mathsf g)$ is called \emph{causally convex} 
    if $J(x,y) \subseteq C$ whenever $x$ and $y$ are in $C$. 
    Equivalently, for all $x,y \in C$, any causal curve segment from $x$ to $y$ is entirely contained in $C$.
\end{Definition}

Examples of causally convex sets are chronological diamonds $I(x,y)$ and causal diamonds $J(x,y)$. In particular, every point in a globally hyperbolic spacetime has a neighbourhood basis of causally convex sets.

\begin{Definition}[Causal $\lambda$-convexity and concavity]
    Let $(\mathcal M,\mathsf g)$ be a globally hyperbolic spacetime and $C \subseteq \mathcal M$ be an open causally convex set. Given $\lambda \in \R$ we say that a function $f : C \rightarrow \R$ is \emph{causally $\lambda$-convex on $C$} provided, for all causal geodesics $\sigma : [0,1] \rightarrow C$ and for every $t \in [0,1]$, we have that
    \begin{equation} \label{eq: definition of concavity}    f(\sigma_t) \leq (1-t)f(\sigma_0) + tf(\sigma_1) - \frac{\lambda}{2} t(1-t) L_\mathsf g(\sigma)^2.
    \end{equation}
where we recall that $L_\mathsf g(\sigma)$ denotes the Lorentzian arclength of $\sigma$. Equivalently, the second derivative of $f$ satisfies $\frac{d^2}{dt^2} f( \sigma (t)) \geq \lambda L_{\mathsf g}(\sigma)^2$ in the distributional sense. A function $f:C \rightarrow \R$ is called \emph{causally $\lambda$-concave} if instead the reverse inequality holds in \eqref{eq: definition of concavity}.
\end{Definition}

Observe that, to check causal $\lambda$-convexity of a continuous $f: C \to \R$  (resp.\ causally $\lambda$-concavity), it is enough to check the convexity inequality along timelike geodesics which are parametrized proportional to arclength, since null geodesics $\sigma:[0,1] \to C$ can be obtained as pointwise limits of such timelike ones. Moreover, observe that causal $\lambda$-convexity with $\lambda > 0$ implies causal convexity (i.e., causal $\lambda$-convexity with $\lambda = 0$).

\begin{Theorem}[Quantitative causal $\lambda$-convexity of modified time separations]\label{Theorem: convexityoftimeseparation}
Let $(\mathcal M,\mathsf g)$ be a globally hyperbolic spacetime with timelike sectional curvature bounded below by $-\delta \leq 0$ and let $\Upsilon:[-L,L] \to \mathcal M$ be a timelike maximizing geodesic of length $2L$. Then, for a fixed constant $0 < R < \frac{L}{2}$, the Lipschitz continuous causal function $\mathrm{md}^{-\delta}(\ell(\Upsilon_{-L}, \cdot))$ and the Lipschitz continuous anti-causal function $\mathrm{md}^{-\delta}(\ell(\cdot,\Upsilon_L))$ are causally $\lambda$-convex on the open connected causally convex set $I(\Upsilon_{-R},\Upsilon_R)$, where 
\begin{equation}\label{md}
    \mathrm{md}^{-\delta}(t):=\begin{cases}
        \frac{\cosh(\sqrt{\delta} t) - 1}{\delta}, & {\delta > 0},\\
        \frac{t^2}{2}, & \delta = 0,
    \end{cases}
\end{equation}
and
\begin{equation}
    \lambda = \lambda(\delta, L \, | \, R) = \cosh(\sqrt{\delta} (L-R)).
\end{equation}
In fact, the two functions are even causally $\lambda$-convex (with the same $\lambda$ as above) on the larger open causally convex set $\{\ell(\Upsilon_{-L}, \cdot) > L - R\} \cap \{\ell(\cdot,\Upsilon_L) > L - R\} \subseteq I(\Upsilon_{-L},\Upsilon_L)$. Moreover, $\lambda(\delta,L \, |R) \downarrow 1$ as $(\delta, L^{-1}) \to 0$ if $\sqrt{\delta} L \to 0$.
\end{Theorem}

\begin{proof}
    This is an immediate consequence of Theorem \ref{Theorem: Toponogov convexity version}, noting that the values of $\ell(\Upsilon_{-L},\cdot)$ and $\ell(\cdot,\Upsilon_L)$ are contained in the interval $(L - R, L + R)$ on $I(\Upsilon_{-R},\Upsilon_R)$. Thus, writing 
    \[
    f_\delta(t):=\mathrm{md}^{-\delta}(\ell(\Upsilon_{-L},\sigma_t))
    \]
    (or $f_\delta(t):=\mathrm{md}^{-\delta}(\ell(\sigma_t,\Upsilon_L))$) for any timelike geodesic $\sigma:[0,1] \to I(\Upsilon_{-R},\Upsilon_R)$, we have that
    \begin{equation}
        f_\delta'' \geq L_\mathsf g(\sigma)^2(1 {+} \delta f_\delta) \geq L_\mathsf g(\sigma)^2 (1 {+} \delta \mathrm{md}^{-\delta}(L-R)) = L_\mathsf g(\sigma)^2 \lambda.
    \end{equation}
    The same estimates work on the larger causally convex set $\{\ell(\Upsilon_{-L}, \cdot) > L - R\} \cap \{\ell(\cdot,\Upsilon_L) > L - R\} \cap I(\Upsilon_{-L},\Upsilon_L)$.
\end{proof}

\begin{Definition}[Causal subdifferential] Let $(\mathcal M,\mathsf g)$ be a globally hyperbolic spacetime, $C \subseteq \mathcal M$ an open and causally convex subset, and $f: C \rightarrow \mathbb R$ be a causally convex function. For $x \in C$ we define the \emph{causal subdifferential of $f$ at $x$} as the set
    \[
    \partial_\leq f(x) := \{ \omega \in  T^*_x \mathcal M 
    : f(x) + \omega[v] \leq f(\exp_x(v)) \; \forall \text{ causal } v \in \exp_x^{-1}(C)\}.
    \]
\end{Definition}

\begin{Lemma}[Monotonicity of the causal subdifferential]\label{Lemma: Monotonicitysupergradient}
Let $(\mathcal M,\mathsf g)$ be globally hyperbolic, $C \subseteq \mathcal M$ be open and causally convex, $f:C \to \R$ be $\lambda$-convex causal (or anti-causal), and $\sigma:[0,1] \to C$ be a maximizing timelike geodesic connecting $x:=\sigma_0$ to $y:= \sigma_1$. Then any $v \in \partial_\leq f(x)$ and $w \in \partial_\leq f(y)$ satisfy
\begin{equation}
    w[\sigma'_1] - v[\sigma'_0] \geq \lambda \ell(x,y)^2.
\end{equation}
\end{Lemma}
\begin{proof}
    Let us write $\sigma_t = \exp_x(t\sigma'(0))$ for $t \in [0,1]$. Using the definition of causal subdifferential as well as causal $\lambda$-convexity, we get
    \begin{align*}
        v[\sigma'_0] \leq \limsup_{t \downarrow 0} \frac{f(\sigma_t) - f(\sigma_0)}{t} \leq f(y) - f(x) - \frac{\lambda}{2} \ell(x,y)^2.
    \end{align*}
    Similarly, considering the past reparametrization of $\sigma$, we obtain
    \begin{align*}
        w[\sigma'_1] \geq f(y) - f(x) + \frac{\lambda}{2}\ell(x,y)^2.
    \end{align*}
    Adding these two estimates yields the claim.
\end{proof}

Basic properties of the causal subdifferential are collected in \cite{BraunGigliMcCannVinciniZanardini}. We recall that if the causally convex function $f$ is smooth at $x\in M$ then the causal subdifferential at $x$ is given by the singleton $\{{d} f_x \}$. The Lorentzian setting requires our gradient flows to be curves of `gentlest ascent' (rather than the `steepest descent' more familiar from Riemannian geometry). For $v \in T_xM$, let $v^\flat(\cdot) := g(\cdot, v) \in T_x^*M$ denote the musical isomorphism.
\begin{Theorem}[Existence and stability of gradient flows]
\label{Theorem: local existence}
    Let $(\mathcal M,\mathsf g)$ be a globally hyperbolic spacetime. Let $C \subseteq \mathcal M$ be a connected causally convex open subset and $f: C \rightarrow \mathbb R$ a $c$-steep causally convex function for some $c > 0$. Then $f$ is locally Lipschitz on $C$ and for every $x \in C$, there exists a unique locally Lipschitz causal curve $\gamma:[0,T) \to C$ for some maximal $T>0$ such that
    \begin{itemize}
        \item[i)] $\dot \gamma_t ^{\flat} \in \partial_{\leq} f(\gamma_t)$ for a.e. $t \in [0,T)$ and $\gamma_0=x$, and
        \item[ii)] $ \frac{d}{ds} \big |_{s=t} \, f(\gamma_s) = \| \dot \gamma_t \|_\mathsf g^2 = \left ( \inf_{y \gg \gamma_t} \frac{f(y) - f(\gamma_t)}{\ell(\gamma_t,y)} \right )^2$ for a.e. $t \in [0,T)$.
    \end{itemize}
    Moreover, if $C$ has compact closure in $M$ we have that $\gamma_T:=\lim_{t \uparrow T} \gamma_t \in \partial C$. Such a curve $\gamma$ is called the \emph{gradient flow of $f$ starting at $x$}.

    In addition, the gradient flow is stable in the following sense: If $(x_n) \downarrow x$ in $C$, i.e., $x_{n+1} \leq x_n$ for all $n$ and $x_n \to x$, then the (unique) gradient flows of $f$ starting at $x_n$ converge locally uniformly to the (unique) gradient flow of $f$ starting at $x$.
\end{Theorem}

\begin{proof}
    Existence and uniqueness are proved 
    via the minimizing movements scheme. The unique solution is characterized
    via the energy dissipation equality (EDE) and evolution variation inequality (EVI) formulations of gradient flows, cf. \cite{BraunGigliMcCannVinciniZanardini}. The claim that $\gamma_T \in \partial C$ follows by contradiction. The stability is a simple consequence of the limit curve theorem and uniqueness, cf.\  \cite[Cor.\ 4.22]{BraunGigliMcCannVinciniZanardini}.
\end{proof}

\section{The Lorentzian almost-splitting theorem}
\label{Section: Almost splitting}

This is the main section of the article and is dedicated to the proof of our main theorem, which we first recall for convenience.

\subsection{The main result}
\label{subsection: main result}

\begin{Theorem}[Almost-splitting under timelike sectional curvature bounds]
\label{Theorem: main theorem}
Fix $R > 0$. Then for every $\delta \geq 0$ and $L \gg R$ there exists a quantity $\Psi = \Psi(\delta,L \, | R)$ with $\Psi(\delta,L \, | R) \to 0$ as $(\delta, L^{-1}) \to 0$, such that the following holds: If $(\mathcal M,\mathsf g)$ is a globally hyperbolic spacetime with timelike sectional curvature bounded below by $-\delta \leq 0$ and there exists a maximizing timelike geodesic $\Upsilon:[-L,L] \to \mathcal M$ parametrized by Lorentzian arclength, then there exists a pointed metric space $(X,\mathsf d, x_0)$ such that
\begin{equation}
    {\mathsf d}_{LGH}(J^\mathcal M(\Upsilon_{-R},\Upsilon_R), J^{\R \times X}((-R,x_0),(R,x_0))) \leq \Psi(\delta,L\,|\,R).
\end{equation}
\end{Theorem}

\begin{Remark}[Independence of $\Psi$ on $\dim \mathcal M$]
Let us note here that, unlike in the Riemannian result by Cheeger--Colding \cite{CheegerColding96}, our error term $\Psi$ is independent of the spacetime dimension $n = \dim \mathcal M$. This is due to our assumption of sectional rather than Ricci curvature bounds, so our comparison-geometric estimates always deal with $2$- (rather than $n$-) dimensional model spaces.
\end{Remark}

\subsection{Expansion rate estimates, monotonicity formulas and reparametrized flows}
\label{Subsection: expansion estimates, monotonicity, reparam flows}

This and subsequent (sub)sections are dedicated to the proof of our main result. Notation will remain uniform throughout.

We fix a globally hyperbolic spacetime $(\mathcal M,\mathsf g)$ and $R = 1$ in the formulation of the main theorem (by scaling, this yields the claim for general $R$) and the timelike lower sectional curvature bound ${-}\delta < 0$ (for simplicity; in order not to write two sets of formulas) and $L \gg 20$ large enough. In addition, let $\delta L^4$ be small\footnote{This does not restrict the generality of our result. This is because the error function $\Psi(\delta,L \, | R)$ can be chosen to be monotonically {in}creasing in $\delta$ and $L^{-1}$, thus if we show our result under the scaling assumption $\delta L^4 \to 0$, it follows that in fact $\Psi \to 0$ regardless of how $\delta \to 0$ and $L \to \infty$.}, and fix a maximizing timelike geodesic $\Upsilon:[-L,L] \to \mathcal M$ parametrized by arclength. {We would like to study the gentlest ascent curves of} 
\begin{equation}\label{F and G}
F(x):=\ell(\Upsilon_{-L},x) \mbox{\rm\ and}\  G(x):=\ell(x,\Upsilon_L).
\end{equation}
However, it is $x \mapsto \mathrm{md}^{-\delta}(F(x))$ and $x \mapsto \mathrm{md}^{-\delta}(G(x))$ 
which Theorem \ref{Theorem: convexityoftimeseparation} asserts to be causally $\lambda$-convex on 
\begin{equation}
\label{eq: H}
H:=\{F > L - 20\} \cap \{G > L -20\} \subseteq I(\Upsilon_{-L},\Upsilon_L) \, ,
\end{equation}
where $\lambda =\cosh(\sqrt{\delta}(L-20)) \geq 0$. {Starting from each $x \in H$,
Theorem \ref{Theorem: local existence} therefore provides} a locally Lipschitz gradient flow curve $\gamma:[0,T) \to H$ of $\mathrm{md}^{-\delta} \circ F$, with {$\gamma_0=x$ and} $\gamma_T \in \partial H$. {Heuristically,}
$\bar \gamma_t:=\gamma_{g(t)}$ {satisfies $
\bar \gamma_t' = \nabla F(\bar \gamma_t)$,} where $g$ is the unique solution of {the ordinary differential equation}
\begin{equation}
\label{eq: reparametrization}
    g'(t) = \frac{\sqrt{\delta}}{\sinh(\sqrt{\delta} F(\gamma_{g(t)}))} = \frac{1}{(\mathrm{md}^{-\delta})'(F(\gamma_{g(t)}))}, \quad g(0) = 0.
\end{equation}

Observe that this is an ODE with Lipschitz right hand side, thus there is a unique solution $g$ as claimed. Moreover, notice that $g' > 0$.

\begin{Proposition}[{Synchronous} expansion rate estimate for causally $\lambda$-convex gradient flows]
\label{Prop: first expansion rate estimate}
Let $C$ be open, connected and causally convex, $f:C \to \R$ a causally $\lambda$-convex $c$-steep function for some $c > 0$ and $\lambda \in \R$. Consider two gradient flow curves $\alpha,\beta:[0,T] \to C$ {of $f$} with $\alpha_0 \leq \beta_0$, and $[0,T]$ a subset of the intersection of their maximal domains. Then for all $t \in [0,T]$,
\begin{equation}\label{first}
    \ell(\alpha_t,\beta_t)^2 \geq e^{2\lambda t} \ell(\alpha_0,\beta_0)^2 \geq (1+ 2\lambda t) \, \ell(\alpha_0,\beta_0)^2.
\end{equation}
Moreover, if $\lambda \geq 0$, then
\begin{equation*}
    \ell(\alpha_t,\beta_t)^2 \geq (1+ 2\lambda t)\ell(\alpha_0,\beta_0)^2 \geq 0
\end{equation*}
for all $t \in [0,T]$. Thus, if $\lambda \geq 0$, then $\alpha_t \leq \beta_t$ for every $t \in [0,T]$.
\end{Proposition}
\begin{proof}
    We will show the claim assuming that $\alpha_0 \ll \beta_0$, the case $\alpha_0 \leq \beta_0$ then follows from the stability of the gradient flow. Since $\alpha,\beta$ are (locally Lipschitz) continuous and $\alpha_0 \ll \beta_0$, there exists $t^* > 0$ such that $\alpha_t \ll \beta_t$ for all $t \in [0,t^*)$. For each $t \in [0,t^*)$, let $\sigma^{(t)}:[0,1] \to C$ be a timelike maximizing geodesic connecting $\alpha_t = \sigma^{(t)}_0$ to $\beta_t:=\sigma^{(t)}_1$. Denote by $m^{(t)}:=\sigma^{(t)}_{1/2}$ the midpoint of $\sigma^{(t)}$ and note that $m^{(t)}$ is neither in the future causal cut locus of $\alpha_t$ nor in the past causal cut locus of $\beta_t$. By construction, $\alpha_t \ll m^{(t)} \ll \beta_t$ and
    \begin{align}\label{eq: equalitymidpoint}
        \ell(\alpha_t,\beta_t) = \ell(\alpha_t,m^{(t)}) + \ell(m^{(t)},\beta_t).
    \end{align}
    For each $t \in [0,t^*)$, by openness of $\ll$, we may find $\eta' > 0$ such that $t + \eta' < t^*$ and $\alpha_{t+\eta} \ll m^{(t)} \ll \beta_{t+\eta}$ for all $0 < \eta \leq \eta'$, and the reverse triangle inequality yields
    \begin{align}\label{eq: inequalitymidpoint}
    \ell(\alpha_{t+\eta},\beta_{t+\eta}) \geq \ell(\alpha_{t+\eta},m^{(t)}) + \ell(m^{(t)}, \beta_{t+\eta}).
    \end{align}
    Subtracting \eqref{eq: equalitymidpoint} from \eqref{eq: inequalitymidpoint} and dividing by $\eta$ yields
    \begin{align}\label{eq: midpointdifferencequotient}
        \frac{\ell(\alpha_{t+\eta},\beta_{t+\eta}) - \ell(\alpha_t,\beta_t)}{\eta} \geq \frac{\ell(\alpha_{t+\eta},m^{(t)}) - \ell(\alpha_t,m^{(t)})}{\eta} + \frac{\ell(m^{(t)},\beta_{t+\eta}) - \ell(m^{(t)},\beta_t)}{\eta}.
    \end{align}
    Since $m^{(t)}$ is neither a future cut point of $\alpha_t$ nor a past cut point of $\beta_t$, the functions $s \mapsto \ell(\alpha_s,m^{(t)})$ and $s \mapsto \ell(m^{(t)},\beta_s)$ (defined for $s$ close to $t$) are differentiable at $s=t$. Thus, by taking the $\liminf$ as $\eta \downarrow 0$, the terms on the right-hand side of \eqref{eq: midpointdifferencequotient} become actual limits. In particular, for all $t$ at which $\alpha$, $\beta$, and $t \mapsto \ell(\alpha_t,\beta_t)$ are differentiable (by Lipschitz continuity, these are almost all $t$), {
 the first variation formula (cf. \cite[p.\ 264]{bookONeill}) yields}
    \begin{align}
        &\lim_{\eta \downarrow 0} \frac{\ell(\alpha_{t+\eta},m^{(t)}) - \ell(\alpha_t,m^{(t)})}{\eta} = -\frac{1}{\ell(\alpha_t,\beta_t)} \mathsf g((\sigma^{(t)})'(0),\alpha'_t),\\
        &\lim_{\eta \downarrow 0} \frac{\ell(m^{(t)},\beta_{t+\eta}) - \ell(m^{(t)},\beta_t)}{\eta} = \frac{1}{\ell(\alpha_t,\beta_t)} \mathsf g((\sigma^{(t)})'(1),\beta'_t).
    \end{align}
    Using this, together with the monotonicity property from Lemma \ref{Lemma: Monotonicitysupergradient}, noting that $(\alpha'_t)^{\flat} \in \partial_\leq f(\alpha_t)$ and $(\beta'_t)^{\flat} \in \partial_\leq f(\beta_t)$, we obtain for almost every
    $t \in [0,t^*)$ that
    \begin{align}
        \lim_{\eta \downarrow 0} \frac{\ell(\alpha_{t+\eta},\beta_{t+\eta})- \ell(\alpha_t,\beta_t)}{\eta} &\geq \frac{\mathsf g((\sigma^{(t)})'_1,\beta'_t)-\mathsf g((\sigma^{(t)})'_0,\alpha'_t)}
        {\ell(\alpha_t,\beta_t)} \\
        &= \frac{(\beta'_t)^{\flat}[(\sigma^{(t)})'_1] - (\alpha'_t)^{\flat}[(\sigma^{(t)})'_0]}{\ell(\alpha_t,\beta_t)}\\
        &\geq \lambda \ell(\alpha_t,\beta_t).
    \end{align}
    Integrating this inequality yields the claim for $t \in [0,t^*)$. Continuity yields the claim for $t = t^*$, then a connectedness argument yields it for all $t \in [0,T]$.
\end{proof}

\begin{Lemma}[Subgradient of the Lorentzian distance to a point]
\label{L:gradient estimate}
    \label{variation inequality} Let $(\mathcal M,\mathsf g)$ be a globally hyperbolic spacetime, let $x \ll y$ be any two points, and let $\sigma:[0,1] \to \mathcal M$ be a maximizing timelike geodesic from $x$ to $y$. 
    Let $\gamma:[0,T] \to \mathcal M$ be any future causal curve starting at $y$ which is right-differentiable at $t= 0$, then 
    \begin{equation}\label{gradient estimate}
    \liminf_{t \rightarrow 0^{+}} \frac{\ell(x, \gamma_t)- \ell(x,y)}{t} \geq \frac{ \mathsf g(\sigma'_1, \gamma_0')}{ \ell(x,y)}.
    \end{equation}
    In particular,   
    \begin{equation}
        \frac{d}{dt}\bigg|^+_{t = 0} \frac{1}{2}\ell(x,\gamma_t)^2 \geq \mathsf g(\gamma_0', \sigma'_1).
    \end{equation}
\end{Lemma}

\begin{proof}
    If $x$ does not lie in the past causal cut locus of $y$, then this result holds with an equality instead of an inequality by the first variation formula, see \cite[p 264]{bookONeill}.
    If $x$ is in the past cut locus of $y$ then we have that $\sigma_s$ does not lie in the future cut locus of $x$ for $s \in [0,1)$. Hence we use reversed triangle inequality
    \[
    \ell(x, \gamma_t) \geq \ell( x, \sigma _s) + \ell(\sigma _s, \gamma_t).
    \]
    Now subtract on both sides $\ell(x, y)$ and use the fact that $\sigma$ is a maximizing geodesic to obtain
    \[
    \frac{\ell(x, \gamma_t) - \ell(x, y)}{t} \geq \frac{\ell(x, \sigma_s) + \ell(\sigma_s, \gamma_t) - \ell(x, y)}{t} = \frac{\ell(\sigma_s, \gamma_t) - \ell(\sigma_s, y) }{t}.
    \]
    Take the $\liminf$ in $t \rightarrow 0^+$, notice that the right-hand side is differentiable since $y$ is not in the future cut locus of $\sigma_s$ and $\gamma_t$ is right differentiable at $t= 0$, and {the first variation formula yields}    
    \begin{equation*}
    \liminf_{t \rightarrow 0^+} {\frac{\ell(x, \gamma_t) - \ell(x, y)}{t}} \geq  \frac{\mathsf g(-\exp_y^{-1} (\sigma_s),\gamma'_0)}{\mathsf g(-\exp_y^{-1} (\sigma_s), -\exp_y^{-1}(\sigma_s))^{1/2}}.
    \end{equation*}
    Letting $s \rightarrow 1^-$ on the right-hand side, the conclusion follows. The claim about $\ell(x,\gamma_t)^2$ easily follows from the above together with the chain rule.
\end{proof}

\begin{Proposition}[{Down}stream expansion rate estimate for $\mathrm{md}^{-\delta} \circ F$]
\label{Prop: Second mainstream expansion}
Let $(\mathcal M,\mathsf g)$ be globally hyperbolic with timelike sectional curvature bounded below by $-\delta \leq 0$, and let $\Upsilon:[-L,L] \to \mathcal M$ be a maximizing timelike geodesic parametrized by proper time. Let $F(x):=\ell(\Upsilon_{-L},x)$ and let $\gamma^1,\gamma^2:[0,T] \to H$ be two gradient flow curves of $\mathrm{md}^{-\delta} \circ F$. Suppose that $\gamma^1_0 \leq \gamma^2_0$. Then for every $t \in [0,T]$,
\begin{equation}
\begin{split}
    \ell(\gamma_0^1,\gamma_t^2)^2 & \geq  \ell(\gamma_0^1,\gamma^2_0)^{2} + 2\int_0^t \int_0^s (\mathrm{md}^{-\delta})'(F(\gamma^2_r))^2 \, dr \, ds \\ & \qquad \qquad +2t(\mathrm{md}^{-\delta}(F(\gamma^2_0)) - \mathrm{md}^{-\delta}(F(\gamma^1_0))) + \lambda \int_0^t \ell(\gamma^1_0,\gamma^2_s)^{2} \, ds.
    \end{split}
\end{equation}
\end{Proposition}

\begin{proof}
    Again we may suppose that $\gamma_0^1 \ll \gamma_0^2$. The case $\gamma_0^1 \leq \gamma_0^2$ follows from gradient flow stability. 
    Fix $f(x):=\mathrm{md}^{-\delta}(F(x))$. Since $r \mapsto f(\gamma^2_r)$ is increasing, then
    \begin{equation*}
        f(\gamma_s^2) - f(\gamma^2_0) \geq \int_0^s \frac{d}{dr}  f(\gamma^2_r) \, dr.
    \end{equation*}
    By properties of the gradient flow, the fact that $F$ is $1$-steep, and since $f(\gamma^2_r)$ is differentiable a.e., it follows for a.e. $r$ that
    \begin{equation*}
        \frac{d}{dr} f(\gamma_r^2) = (\mathrm{md}^{-\delta})'(F(\gamma^2_r)) \frac{d}{dr} \, F(\gamma^2_r) \geq (\mathrm{md}^{-\delta})'(F(\gamma^2_r)) \|\dot \gamma^2_r\| \geq ((\mathrm{md}^{-\delta})'(F(\gamma^2_r))^2.
    \end{equation*}
    This gives the estimate
    \begin{equation}
    \label{eq: fgamma^2_1 - fgamma^2_0 estimate}
        f(\gamma_s^2) - f(\gamma^2_0) \geq \int_0^s ((\mathrm{md}^{-\delta})'(F(\gamma^2_r))^2 \, dr. 
    \end{equation}
    {Now} connect $\gamma_0^1$ to $\gamma^2_s$ by a maximizing timelike geodesic $\sigma^{(s)}:[0,1] \to H$. Using the $\lambda$-convexity of $f$ along $\sigma^{(s)}$ and the fact that $(\dot \gamma_s^2)^{\flat} \in \partial_\leq f ( \gamma^2_s)$, we have that
    \begin{equation}
        (\dot \gamma_s^2)^{\flat} [\dot \sigma^{(s)}_1] \geq \lim_{\varepsilon \rightarrow 0} \frac{f( \sigma^{(s)}_{1})- f(\sigma^{(s)}_{1- \varepsilon} )}{\varepsilon} \geq f(\sigma_1^{(s)}) - f(\sigma_0^{(s)} ) + \frac{\lambda}{2} \ell(\gamma_0^1, \gamma_s^2)^2 ,
    \end{equation}
    and so
    \begin{equation}\label{EVI 2a}
        g( \dot \gamma_s^2 , \dot \sigma^{(s)}_1 ) \geq f(\gamma_s^{2} ) - f(\gamma_0^1) + \frac{\lambda}{2} \ell(\gamma_0^1, \gamma_s^2)^2
    \end{equation}
    Now, using Lemma \ref{variation inequality}, we obtain for a.e. $s \in [0,T]$ that
    \begin{equation} \label{eq: EVI 2}
        \dfrac{d}{dt}\bigg{|}^+_{t=s} \frac{1}{2} \ell(\gamma^1_0, \gamma_t^2)^2 \geq \mathsf g( \dot \gamma_s^2 , \dot \sigma^{(s)}_1 )  .
    \end{equation}
    Next we can {combine \eqref{EVI 2a}}, \eqref{eq: EVI 2}, and \eqref{eq: fgamma^2_1 - fgamma^2_0 estimate} to get
    \begin{equation}
        \dfrac{d}{dt}\bigg{|}^+_{t=s} \frac{1}{2} \ell(\gamma^1_0, \gamma_t^2)^2 \geq f(\gamma_0^2) - f(\gamma_0^1) +\int_0^s ((\mathrm{md}^{-\delta})'(F(\gamma^2_r))^2 \, dr + \frac{\lambda}{2} \ell(\gamma_0^1, \gamma_s^2)^2 .\label{eq3.20}
    \end{equation}
    The conclusion follows by integrating \eqref{eq3.20} on $s \in [0,t]$ and noticing that 
    \[
    \frac{1}{2} \ell(\gamma_0^1, \gamma_t^2)^2 - \frac{1}{2} \ell(\gamma_0^1, \gamma_0^2)^2 \geq \int_0^t \dfrac{d}{d\tau}\bigg{|}^+_{\tau=s} \frac{1}{2} \ell(\gamma^1_0, \gamma_{\tau}^2)^2 ds
    \]
    holds by monotonicity of $s \mapsto \frac{1}{2} \ell(\gamma_0^1, \gamma_s^2)^2$.    
\end{proof}

\begin{Proposition}[Upstream expansion rate estimate for $\mathrm{md}^{-\delta} \circ F$]
\label{Prop: Second upstream expansion}
In the setting of Proposition \ref{Prop: Second mainstream expansion}, let $\gamma^1,\gamma^2:[0,T] \to H$ be two gradient flow curves of $\mathrm{md}^{-\delta} \circ F$,
and suppose for some $t \in [0,T]$ we have that $\gamma^1_t \leq \gamma^2_0$. Then
\begin{equation}
\begin{split}
    \ell(\gamma_t^1,\gamma_0^2)^2 & \geq \ell(\gamma_0^1,\gamma_0^2)^2 + 2 \int_0^t \int_0^s (\mathrm{md}^{-\delta})'(F(\gamma^1_r))^2 \, dr ds \\
    & \qquad \qquad {-} 2t(\mathrm{md}^{-\delta}(F(\gamma_0^2)) - \mathrm{md}^{-\delta}(F(\gamma_0^1)))
    + \lambda \int_0^t \ell(\gamma^1_s,\gamma_0^2)^2 \, ds.
\end{split}
\end{equation}
\end{Proposition}
\begin{proof}
    The proof is entirely analogous to that of Proposition \ref{Prop: Second mainstream expansion}, upon noticing that $t > 0$ can be assumed without loss of generality (the claim holds trivially for $t = 0$).
\end{proof}

\begin{Proposition}[{Asynchronous} expansion rate estimate for $\mathrm{md}^{-\delta} \circ F$]
\label{Prop: Final second expansion}
In the setting of Proposition \ref{Prop: Second mainstream expansion}, suppose $\gamma^1,\gamma^2:[0,T] \to H$ are two gradient flows with $\gamma_0^1 \leq \gamma_0^2$. 
For all $s,t \in [0,T]$. If $s\le t$ then
\begin{equation}\label{final 2a}
    \begin{split}\ell(\gamma^1_s,\gamma^2_t)^2 &\geq e^{2\lambda s} \bigg(\ell(\gamma_0^1,\gamma_0^2)^2 + 2 \int_0^{t-s} \int_0^r (\mathrm{md}^{-\delta})'(F(\gamma^2_\xi))^2 \, d\xi \, dr  \\
        &\qquad   +2(t-s) (\mathrm{md}^{-\delta}(F(\gamma_0^2)) - \mathrm{md}^{-\delta}(F(\gamma_0^1))) + \lambda \int_0^{t-s} \ell(\gamma_0^1,\gamma_\xi^2)^2 \, d\xi\bigg) .
    \end{split}
\end{equation}
If instead $s\geq t$ but $\gamma_{s-t}^1 \leq \gamma_0^2$ then
\begin{equation}\label{final 2b}
    \begin{split}
        \ell(\gamma^1_s,\gamma^2_t)^2 &\geq e^{2\lambda t} \bigg(\ell(\gamma_0^1,\gamma_0^2)^2 + 2 \int_0^{s-t} \int_0^r (\mathrm{md}^{-\delta})'(F(\gamma^1_\xi))^2 \, d\xi \, dr  \\
        &\qquad   +2(s-t) (\mathrm{md}^{-\delta}(F(\gamma_0^1)) - \mathrm{md}^{-\delta}(F(\gamma_0^2))) + \lambda \int_0^{s-t} \ell(\gamma_\xi^1,\gamma_0^2)^2 \, d\xi\bigg) .
    \end{split}
\end{equation}
\end{Proposition}
\begin{proof}
    If $t\geq s$ we get that $\gamma_0^1 \leq \gamma_0^2 \leq \gamma_{t-s}^2$, so the {synchronous} expansion rate estimate \eqref{first} based at the points $\gamma_0^1 \leq \gamma_{t-s}^2$  gives 
    \[
    \ell( \gamma_s^1, \gamma_t^2)^2 \geq e^{2\lambda s} \ell (\gamma_0^1, \gamma_{t-s}^2)^2 .
    \]
    Now \eqref{final 2a} follows by estimating  $\ell(\gamma_0^1,\gamma_{t-s}^2)^2$
    from below using Proposition \ref{Prop: Second mainstream expansion}. The second claim \eqref{final 2b} is argued analogously, using the {synchronous} expansion rate estimate in conjunction with Proposition \ref{Prop: Second upstream expansion}.
\end{proof}

\begin{Corollary}[{Simplified asynchronous} expansion rate estimates]
\label{Cor: Compact rewriting of second expansion estimates}
Let $(\mathcal M,\mathsf g)$ globally hyperbolic with timelike sectional curvature bounded below by $-\delta < 0$, $\Upsilon:[-L,L] \to \mathcal M$ a timelike maximizing geodesic parametrized by proper time. Let $F(x):=\ell(\Upsilon_{-L},x)$, and $\gamma^1,\gamma^2:[0,T] \to H $ be two gradient flow curves of $\mathrm{md}^{-\delta} \circ F$ satisfying $\gamma^1_0 \leq \gamma^2_0$. Let $s,t \in [0,T]$. Assume either $s \leq t$ or $s > t$ and $\gamma^1_{s - t} \leq \gamma^2_0$. Then
\begin{equation}\label{eq3.24}
    \begin{split}
        \ell(\gamma^1_s,\gamma^2_t)^2 &\geq e^{2 \lambda \min \{ s,t \} } \bigg( \ell(\gamma_0^1,\gamma_0^2)^2 + (t-s)^2(\mathrm{md}^{-\delta})'(F(\gamma_0^1))^2 \\
        &\qquad \qquad  + 2(t-s) (\mathrm{md}^{-\delta}(F(\gamma_0^2)) - \mathrm{md}^{-\delta}(F(\gamma_0^1)))\bigg).
    \end{split}
\end{equation}
\end{Corollary}
\begin{proof}
    By monotonicity of $(\mathrm{md}^{-\delta})'$ and the fact that $F$ is a causal function and $\gamma^1, \gamma^2$ are causal curves, it follows that for every $s,t \in [0,T]$ (given our assumption $\gamma_0^1 \leq \gamma_0^2$)
    \[
    \min \left \{ (\mathrm{md}^{-\delta})'(F(\gamma^2_s)),(\mathrm{md}^{-\delta})'(F(\gamma^1_t)) \right \} \geq (\mathrm{md}^{-\delta})'(F(\gamma_0^1)).
    \]
    Together with the dropping of the term involving $\lambda$, the stated inequality \eqref{eq3.24} {summarizes essential consequences of} the estimates obtained in Propositions \ref{Prop: Second mainstream expansion}, \ref{Prop: Second upstream expansion} and \ref{Prop: Final second expansion}.
\end{proof}

Let us describe the gradient flow of $\mathrm{md}^{-\delta} \circ F = \mathrm{md}^{-\delta}(\ell(\Upsilon_{-L},\cdot))$ in the special case that the initial point $x$ belongs to the support of the geodesic $\Upsilon$. To this end, fix $\varepsilon \in [-20,20]$, and make the ansatz $\gamma_t = \Upsilon_{h(t)}$, so that $\gamma'_t = h'(t) \Upsilon'_{h(t)}$. Since $F$ is smooth in a neighbourhood of $\Upsilon([-20,20])$, we have that $\gamma'_t = (\mathrm{md}^{-\delta})'(F(\gamma_t)) \cdot \nabla F(\gamma_t) = (\mathrm{md}^{-\delta})'(h(t) + L) \Upsilon'_{h(t)}$. {Finding $\gamma$ therefore amounts to solving the following ordinary differential equation}:
\begin{equation}
\begin{cases}
    h'(t) = (\mathrm{md}^{-\delta})'(h(t) + L), \\ h(0) = \varepsilon.
\end{cases}
\end{equation}
This may be solved explicitly. The solution $h = h_\varepsilon$ is
\begin{equation}\label{hed}
    h_\varepsilon(t) = \frac{2}{\sqrt{\delta}} \,\mathrm{arctanh} \left (e^t \tanh\left( \frac{\sqrt{\delta}(\varepsilon + L)}{2}\right)\right ) - L \, .
\end{equation}
For $\delta>0$, the solution $h_\varepsilon$ blows up at a finite time
\begin{equation}
    t_* = -\log \left( \tanh \left( \frac{\sqrt{\delta}(\varepsilon + L)}{2} \right)\right)
\end{equation}
which diverges as $\delta \downarrow 0$.
Now our choice $\delta \ll L^{-4} \ll 1$ of scaling yields that $\sqrt{\delta}L $ is small. We desire it to be small enough so that $t_*$ is so large that its finiteness never enters into our results. If the lower bound of the curvature is $\delta=0$, then 
\eqref{hed} takes the {simpler} form
\begin{equation*}
h_{\varepsilon}(t) = (\varepsilon + L)e^t -L.
\end{equation*}
For general $\delta$ and $t < t_*$, given our scaling, we have that $h_\varepsilon(t) = (\varepsilon+ L)e^t - L + \Psi(\delta,L)$. Let us note that the gradient flow of $\mathrm{md}^{-\delta} \circ G$ along $\Upsilon$ starting at $\Upsilon_\varepsilon$ is given precisely by $\Upsilon_{-h_{-\varepsilon}(t)}$.

These solutions permit us to write down a monotonicity formula for the gradient flows of the modified time separation function: 

\begin{Lemma}[Monotonicity formula for the modified flow] \label{lem: monotonicity 1}
    Let $x \in J(\Upsilon_{\tilde R}, \Upsilon_{R})$, where $-20 <\tilde R \leq R < 20$ and let $\gamma: [0,T] \rightarrow \mathcal M$ be the gradient flow of $\mathrm{md}^{-\delta}\circ F$ starting at $x$, and $[0,T]$ some compact subinterval in its maximal domain of existence. Then 
    \[
    \gamma_t \in J ( \Upsilon_{h_{\tilde R}(t)}, \Upsilon_{h_{R}(t)} ) 
    \]
    for all $t \in [0,T]$ such that $-20 \leq h_{\tilde R}(t) \leq h_{R}(t) \leq 20$. Analogously, if $\sigma :[0,T] \rightarrow \mathcal M$ is the gradient flow of $\mathrm{md}^{-\delta} \circ G$ starting at $x$, then 
    \[
    \sigma_t \in J ( \Upsilon_{-h_{-\tilde R}(t)}, \Upsilon_{-h_{-R}(t)} ) 
    \]
    for all $t \in [0,T]$ such that $-20 \leq -h_{-\tilde R}(t) \leq -h_{-R}(t) \leq 20$.
\end{Lemma}

\begin{proof}
We argue the result in the case of the gradient flow of $\mathrm{md}^{-\delta} \circ F$ since the other case is analogous. Fix $x \in J(\Upsilon_{\tilde{R}},\Upsilon_R)$. Then,
since the gradient flow curve starting at $\Upsilon_{\tilde R}$ is precisely $\Upsilon_{h_{\tilde{R}}(t)}$ and similarly the flow starting at $\Upsilon_R$ is $\Upsilon_{h_R(t)}$, the claim follows by the {synchronous} expansion rate estimate {Proposition \ref{Prop: first expansion rate estimate}, which ensures monotonicity of the modified flow}.
\end{proof}

\begin{Lemma}[Almost monotonicity for the gradient flow]
\label{Lemma: Almost monotonicity of reparam grad flow}
    Let $x \in J(\Upsilon_{\tilde R},\Upsilon_{R})$, where $-10 <\tilde R \leq R < 10$. Let $\gamma$ be the maximal gradient flow of $\mathrm{md}^{-\delta}\circ F$ in $H$ (where $H$ is defined in \eqref{eq: H}) starting at $x$, and let $\bar\gamma_t:=\gamma_{g(t)}$ be its reparametrization. Then
    \[
    \overline{\gamma}_t \in J( \Upsilon_{\tilde R + t+ \Psi(\delta, L)} , \Upsilon_{R+ t +\Psi(\delta, L)} )
    \]
    for every $t \in [0,4]$.
\end{Lemma}

\begin{proof}
    The reparametrization $g=g(t)$ satisfies 
    {\eqref{eq: reparametrization}.}
    {Since $(\mathrm{md}^{-\delta})' (\cdot)$ 
    is nondecreasing}
    and Lemma \ref{lem: monotonicity 1} 
    {yields $\gamma_{g(t)} \in J(\Upsilon_{h_{\tilde R}(g(t))},\Upsilon_{h_R(g(t))})$, we find}
    \begin{align*}
    \frac{1}{(\mathrm{md}^{-\delta})' (F(\Upsilon_{h_R\circ g(t)}))} \leq g'(t) 
    &\leq \frac{1}{(\mathrm{md}^{-\delta})' (F(\Upsilon_{h_{\tilde R}\circ g(t)}))}.
    \end{align*}
    The 
    {lower bound} can be computed explicitly:
    \begin{align*}
    \frac{1}{(\mathrm{md}^{-\delta})' (F(\Upsilon_{h_R\circ g(t)}))}&=
        \frac{1}{(\mathrm{md}^{-\delta})'(h_{R}(g(t)) + L)} 
        \\& = \sqrt{\delta} \, \frac{1-e^{2g(t) \tanh^2(\sqrt{\delta}(L+{R})/2)}}{ 2e^{g(t)} \tanh(\sqrt{\delta}(L+{R})/2)}
    \end{align*}
    {where the formula ${\rm arctanh} s = \log \sqrt{\frac{1+s}{1-s}}$ and Equation \eqref{hed} have been used.}
    Given that $\sqrt{\delta} L \to 0$, an explicit computation shows that the right hand side is of the form $\frac{1 + \Psi(\delta,L)}{L+R}e^{-g(t)}$. Multiplying by $e^{g(t)}$ and integrating  gives us that
    \begin{equation}
    \label{eq: bounds on g}
        \log \left( 1 + \frac{t + \Psi(\delta,L)}{L + R} \right) \leq g(t) \leq \log \left( 1 + \frac{t + \Psi(\delta,L)}{L + \tilde{R}}  \right).
    \end{equation}
    Using now the monotonicity established in Lemma \ref{lem: monotonicity 1}, we get that $\bar \gamma_t \in J(\Upsilon_{s_-(t)},\Upsilon_{s_+(t)})$, where $s_\pm(t)$ have the claimed form. Let us demonstrate this for $s_-(t)$:
    \begin{align*}
        s_-(t) &= h_{\tilde R}(\log(1 + (t + \Psi(\delta,L))/(L+{R}))) 
        \\&= (L + \tilde R) (1 + (t+ \Psi)/(L+R)) - L + \Psi\\
        &= \tilde R + \frac{L + \tilde R}{L + R}(t + \Psi) + \Psi \\
        &= \tilde R + t + \Psi,
    \end{align*}
    where in the last equality we use $t \in [0,4]$ to collect the term $t(\tilde R - R)/(L + R)$ into the term $\Psi$.
\end{proof}

\begin{Remark}
    If $\delta=0$ the previous lemma gets much simpler and more explicit. Let $\Upsilon_{\tilde{R}} \leq x \leq \Upsilon_{R}$ and $g$ be the reparametrization of the modified flow starting at $x$. Then, following the proof of the lemma
    we get that
    \begin{equation*}
    \frac{1}{(L+R) e^{g(t)}} \leq g'(t) \leq  \frac{1}{(L+ \tilde R) e^{g(t)}}
    \end{equation*}
    and the ODE comparison produces 
    \begin{equation*}
    \log \left ( 1+ \frac{t}{L+R} \right ) \leq g(t) \leq \log \left ( 1+ \frac{t}{L+ \tilde{R}} \right ).
    \end{equation*}
    We remark that the last inequality holds for all $t \leq  4$
    and there is no blow up for the reparametrization. Moreover we get that
    \[
    \overline{\gamma}_t \in J( \Upsilon_{\tilde R + \frac{L+ \tilde R}{L+R} t} , \Upsilon_{R + \frac{L+R}{L+ \tilde R} t})
    \]
    for every $t \leq 4$ and every $x \in J( \Upsilon_{\tilde R}, \Upsilon_{R})$. Notice that this is compatible with the previous lemma since $\frac{L+ \tilde{R}}{L+R}t = t - \frac{R -\tilde R}{L+R}t = t + \Psi(\delta, L)$.
    
\end{Remark}

\begin{Proposition}[{Asynchronous} expansion rate estimate for the reparametrized flow]
\label{Prop: second estimate for reparametrized flow}
Let $(\mathcal M,\mathsf g)$ be globally hyperbolic with timelike sectional curvature bounded below by $-\delta \leq 0$ and let $\Upsilon:[-L,L] \to \mathcal M$ be a timelike maximizing geodesic parametrized by Lorentzian arclength. 
Let $F(x):=\ell(\Upsilon_{-L},x)$, and $\gamma^1,\gamma^2:[0,T] \to \mathcal M$ be two gradient flow curves of $\mathrm{md}^{-\delta} \circ F$ starting in $ I(\Upsilon_{\tilde R},\Upsilon_{R})$, where $-10 < \tilde R \leq R < 10 \le L/2$, and satisfying $\gamma^1_0 \leq \gamma^2_0$. Let $g_1,g_2$ be the reparametrizations of $\gamma^1,\gamma^2$ introduced in \eqref{eq: reparametrization} and write $\bar \gamma^j = \gamma^j \circ g_j$, $j = 1,2$ for the reparametrized curves. Let $s,t \in [0,4]$ such that $\bar \gamma_s \leq \bar \gamma_t$. If $s > t$, suppose in addition that $\bar \gamma^1_{s-t} \leq \bar \gamma^2_0$. Then
\begin{equation}
    \ell(\bar \gamma^1_s,\bar \gamma^2_t)^2 \geq \ell(\gamma_0^1, \gamma_0^2)^2 + (t-s)^2 + 2(t-s)(F(\gamma_0^2)- F(\gamma_0^1)) + \Psi (\delta, L) \, .
\end{equation}
\end{Proposition}
\begin{proof}
    We argue the case $s \leq t$, the argument for $s > t$ (under the additional assumption $\bar \gamma_{s-t} \leq \bar \gamma_t)$ is analogous. We make use of the compactly rewritten {asynchronous} expansion rate estimate given in Corollary \ref{Cor: Compact rewriting of second expansion estimates}. Thus,
    \begin{align*}
        \ell(\bar\gamma^1_s,\bar \gamma^2_t)^2 = \ell(\gamma^1_{g_1(s)},\gamma^2_{g_2(t)})^2 &\geq e^{2 \lambda \min(g_1(s),g_2(t))} \bigg(\ell(\gamma^1_0,\gamma^2_0)^2 + (g_2(t) - g_1(s))^2 (\mathrm{md}^{-\delta})'(F(\gamma^1_0))^2 \\
        &\quad + 2(g_2(t) - g_1(s)) (\mathrm{md}^{-\delta}(F(\gamma_0^2)) - \mathrm{md}^{-\delta}(F(\gamma_0^1))) \bigg).
    \end{align*}
    Then
    {\eqref{eq: bounds on g} yields}
    \begin{equation*}
        g_2(t) - g_1(s) \geq \log\left(1 + \frac{t+\Psi}{L + \tilde R} \right) - \log\left(1 + \frac{s+\Psi}{L + R}\right) = \frac{t}{L + \tilde R} - \frac{s}{L +R} + \frac{\Psi}{L},
    \end{equation*}
    where in the last estimate we used the Maclaurin expansion for $\log (1+x)$. Since {\eqref{eq: reparametrization} gives}
    \begin{equation*}
        (\mathrm{md}^{-\delta})'(F(\gamma_0^1))^2 \geq (\mathrm{md}^{-\delta})'(L +\tilde R)^2 = (L + \tilde R)^2 + \Psi,
    \end{equation*}
    we have that
    \begin{align*}
        (g_2(t) - g_1(s))^2 (\mathrm{md}^{-\delta})'(F(\gamma_0^1))^2 &\geq \bigg( \frac{t}{L + \tilde R} - \frac{s}{L + R} + \frac{\Psi}{L}\bigg)^2 ((L + \tilde R)^2 + \Psi)\\
        &= \bigg(\left(\frac{t}{L + \tilde R} - \frac{s}{L +R}\right)^2 + \frac{\Psi}{L^2}\bigg) ((L+\tilde R)^2 + \Psi)\\
        &= (t -s)^2 + \Psi.
    \end{align*}
    The mixed term can be estimated in a similar way. Finally, noticing that $e^{2 \lambda \min(g_1(s),g_2(t))} = 1 + \Psi$ concludes the proof.
\end{proof}

\subsection{Busemann sum estimates, almost linearity, almost Pythagoras formula}
\label{Subsection: Busemann sum estimates, almost linearity, almost Pythagoras}

\begin{Definition}[Approximate Busemann functions]
Given the timelike maximizer $\Upsilon:[-L,L] \to \mathcal M$, we use $F$ and $G$ from \eqref{F and G} to define the \emph{approximate Busemann functions} $b_L^\pm: J(\Upsilon_{-L},\Upsilon_L) \to \R$ associated with $\Upsilon$ via
\begin{align}
\label{b^+_L}
    &b^+_L(x):=L - \ell(\Upsilon_{-L},x) =: L - F(x),\\
    &b_L^-(x):= L - \ell(x, \Upsilon_L) =: L - G(x).
\label{b^-_L}
\end{align}
\end{Definition}

We remark that our conventions for $b^\pm_L$ are exactly opposite to the ones used for (approximate) Busemann functions in \cite{QuintetI}. Note the reverse triangle inequality implies $b_L^+ + b_L^- \geq 0$ on $J(\Upsilon_{-L},\Upsilon_L)$, and $b_L^\pm$ are Lipschitz on any diamond of the form $J(\Upsilon_{-L + \varepsilon}, \Upsilon_{L - \varepsilon})$ for any $0 < \varepsilon < L$.

\begin{Proposition}[Estimate on the Busemann sum]
\label{Prop: Estimate on Busemann sum}
On $J:=J(\Upsilon_{-20},\Upsilon_{20})$, we have that
\begin{equation}
    0 \leq b^+_L + b^-_L \leq \Psi(\delta,L).
\end{equation}
Moreover, there exists a constant $C > 0$ independent of $L$ and $\delta$ such that
\begin{equation}
    -C \leq L^2 - \frac{F^2}{2} - \frac{G^2}{2} \leq C \qquad \text{on $J$} .
\end{equation}
\end{Proposition}
\begin{proof}
    We restrict ourselves to showing this inequality in the interior $I := I(\Upsilon_{-20},\Upsilon_{20})$ of $J$. The claim for $J$ then follows by continuity. Pick $x \in I$ and let $\sigma:[0,1] \to \mathcal M$ be a maximizing timelike geodesic parametrized proportional to $g$-arclength connecting $\Upsilon_{-20}$ to $x$. By causal $\lambda$-convexity of $F_\delta:=\mathrm{md}^{-\delta} \circ F$ and $G_\delta:=\mathrm{md}^{-\delta} \circ G$ on $I(\Upsilon_{-20},\Upsilon_{20})$, we have that 
    \begin{align}
        F_\delta(x) &\geq F_\delta (\Upsilon_{-20}) + \frac{d}{dt}\bigg|_{t=0} F_\delta(\sigma_t) + \frac{\lambda}{2} \ell(\Upsilon_{-20},x)^2 \nonumber \\
        &= \mathrm{md}^{-\delta}(L - 20) + (\mathrm{md}^{-\delta})'(L - 20) \langle \Upsilon_{-20}', \sigma'_0\rangle  + \frac{\lambda}{2} \ell(\Upsilon_{-20}, x)^2. \label{eq: F_delta expansion}
    \end{align}
    Analogously,
    \begin{align}
    \label{eq: G_delta expansion}
        G_\delta(x) \geq \mathrm{md}^{-\delta}(L + 20) - (\mathrm{md}^{-\delta})'(L + 20) \langle \Upsilon_{-20}',\sigma_0'\rangle + \frac{\lambda}{2} \ell(\Upsilon_{-20}, x)^2.
    \end{align}
    Let us define the quantities $A:=\langle \Upsilon'_{-20},\sigma_0'\rangle$ and $B:=\ell(\Upsilon_{-20},x)^2$. Convexity and the fact that $F(x) \leq L + 20$ imply
    \begin{equation*}
        0 \leq A \leq \frac{\mathrm{md}^{-\delta}(L+20) - \mathrm{md}^{-\delta}(L-20)}{(\mathrm{md}^{-\delta})'(L-20)} \leq 40 \frac{(\mathrm{md}^{-\delta})'(L+20)}{(\mathrm{md}^{-\delta})'(L-20)} \leq 80,
    \end{equation*}
    where the last inequality holds for sufficiently large $L$. Moreover, by the reverse triangle inequality,
    \begin{equation*}
        0 \leq B = \ell(\Upsilon_{-20},x)^2 \leq \ell(\Upsilon_{-20},\Upsilon_{20})^2 = 1600.
    \end{equation*}
    Observe that the bounds on $A$ and $B$ do not depend on $L$ or $\delta$. Expanding \eqref{md} for small $\delta$ shows that $F_\delta = \frac{F^2}{2} + \Psi$, and similarly $G_\delta = \frac{G^2}{2} + \Psi$. It follows (by similarly expanding the right hand sides in \eqref{eq: F_delta expansion} and \eqref{eq: G_delta expansion}) that 
    \begin{align}
        \frac{F(x)^2}{2} &\geq \frac{(L - 20)^2}{2} + (L-20)\langle \Upsilon_{-20}',\sigma_0'\rangle + \frac{\lambda}{2}  \ell(\Upsilon_{-20},x)^2 + \Psi, \label{eq: estimate from below for F}\\
        \frac{G(x)^2}{2} &\geq \frac{(L + 20)^2}{2} - (L+20)\langle \Upsilon_{-20}',\sigma_0'\rangle + \frac{\lambda}{2}  \ell(\Upsilon_{-20},x)^2 + \Psi. \label{eq: estimate from below for G}
    \end{align}
    Upon taking square roots, we obtain
    \begin{align}
        &F(x) \geq \sqrt{(L - 20)^2 + 2(L - 20)A + \lambda B} + \Psi(\delta, L),\\
       &G(x) \geq \sqrt{(L + 20)^2 - 2(L + 20)A + \lambda B} + \Psi(\delta,L).
   \end{align}
    We conclude that
    \begin{align*}
        0 &\leq b_L^-(x) + b_L^+(x) 
        \\&= 2L - F(x) - G(x)\\
        &\leq 2L - \sqrt{(L - 20)^2 + 2(L - 20)A + \lambda B} 
        - \sqrt{(L + 20)^2 - 2(L + 20)A + \lambda B}
         + \Psi(\delta, L)\\
        &= \Psi(\delta,L).
    \end{align*}
    To obtain the additional claim, we sum \eqref{eq: estimate from below for F} and \eqref{eq: estimate from below for G} and simplify:
    \begin{align*}
        L^2 - \frac{F(x)^2}{2} - \frac{G(x)^2}{2} &\leq L^2 -  \frac{(L - 20)^2}{2} - (L-20)A - \lambda  B
        - \frac{(L + 20)^2}{2} + (L+20)A
        + \Psi\\ &\leq C.
    \end{align*}
    For the lower bound, notice that {$F,G \in [L-20,L+20]$ on $J(\Upsilon_{-20},\Upsilon_{20})$. Hence $|F-G|\le 40$ while} $2L \geq F + G$ by the reverse triangle inequality. Thus
    \begin{align*}
        L^2 - \frac{F^2}{2} - \frac{G^2}{2} {=} \underbrace{\bigg( L - \frac{F + G}{2} \bigg)}_{\geq 0} \bigg(L + \frac{F + G}{2}\bigg) - \frac{1}{4} (F - G)^2 \geq -C.
    \end{align*}
\end{proof}

\begin{Proposition}[Almost linearity of the Busemann functions along the flow]
\label{Proposition: Almost linearity}
For every $x \in I(\Upsilon_{-10},\Upsilon_{10})$ and $t \in [0,4]$, the {(reparametrized)} flow $\bar \gamma_t$ of $F$ \eqref{F and G} starting at $x$ satisfies
\begin{equation}
{ 0  \leq F(\bar \gamma_t) - F(x) -t \leq \Psi(\delta,L). }
\label{linearity}
\end{equation}
\end{Proposition}
\begin{proof}
By Lemma \ref{Lemma: Almost monotonicity of reparam grad flow}, the flow curves under consideration remain in
$I(\Upsilon_{-15},\Upsilon_{15})$ for $t\in[0,4]$.
We establish the following pointwise derivative estimates at
an arbitrary point $x\in I(\Upsilon_{-15},\Upsilon_{15})$.

    We proceed by bounding the right derivative $\frac{d}{dt}|^+_{t = 0} F(\bar \gamma_t)$ of $t \mapsto F(\bar \gamma_t)$ at $t = 0$. {The lower bound} is a consequence of $1$-steepness of $F$, {as we now show}:
    \begin{align*}
        \frac{d}{dt}\bigg|^+_{t = 0} F(\bar \gamma_t) \geq \frac{d}{dt}\bigg|^+_{t = 0} \ell(\bar \gamma_0,\bar \gamma_t) = g'(0) \frac{d}{dt}\bigg|^+_{t = 0} \ell(x, \gamma_t).
    \end{align*}
    Observe that the definition of gradient flow 
    {from Theorem \ref{Theorem: local existence}} and $1$-steepness of $F$ give
    \begin{align*}
        \ell(x,\gamma_t) \geq \int_{0}^t \|\dot \gamma_s\|_g \, ds \geq \int_0^t (\mathrm{md}^{-\delta})'(F(\gamma_s)) \, ds \geq t \, (\mathrm{md}^{-\delta})'(F(\gamma_0)).
    \end{align*}
    We conclude that
    \begin{equation*}
        \frac{d}{dt}\bigg|^+_{t = 0} \ell(x, \gamma_t) \geq (\mathrm{md}^{-\delta})'(F(x)).
    \end{equation*}
    Since $g'(0) = (\mathrm{md}^{-\delta})'(F(x))^{-1}$ by definition, it follows that
    \begin{equation*}
    \frac{d}{dt}\bigg|^+_{t = 0} F(\bar \gamma_t) \geq  1.
    \end{equation*}
    {Turning to the upper bound,} let $\eta:[0,F(x)] \to \mathcal M$ be a past directed maximizing $\mathsf g$-arclength parametrized timelike geodesic from $x$ to $\Upsilon_{-L}$. 

    It is possible to extend $\eta$ to a domain $[-1/4,F(x)]$, still
in $\mathsf g$-unit speed parametrization, such that
$\eta_{-1/4}\in I(\Upsilon_{-20},\Upsilon_{20})$.
To see this, smoothly extend $\eta$ by solving the geodesic equation
until its first intersection with $\partial J^-(\Upsilon_{20})$,
say at parameter $-\varepsilon$. Using triangle comparison for $\Delta(\Upsilon_{-L},\eta_{-\varepsilon},
\Upsilon_{20})$, together with the law of cosines, we obtain
\[
25 \le \ell(x,\Upsilon_{20})^2
\le \varepsilon\left[
\frac{(L+20)^2}{F(x)+\varepsilon}-F(x)
\right]+\Psi(\delta,L)
\le 80\varepsilon+\Psi(\delta,L)
\]
for sufficiently large $L$. Here $F(x)+\varepsilon$ is the length of the extended geodesic,
and the last inequality follows from $F(x)\ge L-15$. Thus $\varepsilon>1/4$. We now set $a:=1/1000$. Moreover, notice that
$\eta_a\in I(\Upsilon_{-20},\Upsilon_{20})$.
Indeed, the function
\[
f(t):=\operatorname{md}^{-\delta}
\bigl(\ell(\Upsilon_{-20},\eta_t)\bigr)
\]
is convex by Theorem \ref{Theorem: convexityoftimeseparation} as long as
$\eta_t\in I^+(\Upsilon_{-20})$. Since $x \in I(\Upsilon_{-15},\Upsilon_{15})$ and $\eta_{-1/4} \in I(\Upsilon_{-20},\Upsilon_{20})$, we can estimate
$f(0)\ge 25/2$ and $f(-1/4)\le 800+\Psi(\delta,L)$.
Using the convexity inequality at the points $-1/4,0,t$,
we therefore obtain
\[
f(t)\ge (1+4t)f(0)-4t f(-1/4)
\ge \frac{25}{2}-3150t-\Psi(\delta,L)>0
\qquad\text{for }0\le t\le a,
\]
provided $\delta,L^{-1}$ are small enough. Consequently,
$\eta_t$ cannot meet $\partial I^+(\Upsilon_{-20})$ in this
interval, and hence
$\eta([-a,a])\subset I(\Upsilon_{-20},\Upsilon_{20})$.
Restrict $\eta$ to $[-a,a]$. By convexity, the difference of right and left derivatives
can therefore be estimated as follows:
\begin{align*}
0 &\le
\left.\frac{d}{ds}\right|_{s=0}^{+}
(\operatorname{md}^{-\delta})(F(\eta_s))
-
\left.\frac{d}{ds}\right|_{s=0}^{-}
(\operatorname{md}^{-\delta})(F(\eta_s))
\\
&\le \frac{1}{a}\Big[
-2(\operatorname{md}^{-\delta})(F(x))
+(\operatorname{md}^{-\delta})(F(\eta_{-a}))
+(\operatorname{md}^{-\delta})(F(\eta_a))
\Big].
\end{align*}
Similarly, we obtain the analogous estimate for the difference
of right and left derivatives of
$(\operatorname{md}^{-\delta})(G(\eta_s))$:
\begin{align*}
0 &\le
\left.\frac{d}{ds}\right|_{s=0}^{+}
(\operatorname{md}^{-\delta})(G(\eta_s))
-
\left.\frac{d}{ds}\right|_{s=0}^{-}
(\operatorname{md}^{-\delta})(G(\eta_s))
\\
&\le \frac{1}{a}\Big[
-2(\operatorname{md}^{-\delta})(G(x))
+(\operatorname{md}^{-\delta})(G(\eta_{-a}))
+(\operatorname{md}^{-\delta})(G(\eta_a))
\Big].
\end{align*}
    Since $\eta$ is maximizing and unit speed from $x$ to $\Upsilon_{-L}$, we notice that
    \begin{align*}
     &\frac{d}{ds}\bigg|_{s = 0}^- (\mathrm{md}^{-\delta})(F(\eta_s)) = (\mathrm{md}^{-\delta})'(F(x)) \frac{d}{ds}\bigg|_{s = 0}^-F(\eta_s),
        \\&\frac{d}{ds}\bigg|_{s = 0}^+ (\mathrm{md}^{-\delta})(F(\eta_s)) = -(\mathrm{md}^{-\delta})'(F(x)).
    \end{align*}
It follows, by adding the nonnegative term involving $G$ to
the estimate for $F$, that
\begin{align*}
0 &\le
-(\operatorname{md}^{-\delta})'(F(x))
\left(1+\left.\frac{d}{ds}\right|_{s=0}^{-}F(\eta_s)\right)
\\
&\le \frac{1}{a}\Big[
-2\big[(\operatorname{md}^{-\delta})(F(x))
      +(\operatorname{md}^{-\delta})(G(x))\big]
\\
&\quad
+\big[(\operatorname{md}^{-\delta})(F(\eta_{-a}))
     +(\operatorname{md}^{-\delta})(G(\eta_{-a}))\big]
\\
&\quad
+\big[(\operatorname{md}^{-\delta})(F(\eta_a))
     +(\operatorname{md}^{-\delta})(G(\eta_a))\big]
\Big].
\end{align*}
Expanding at this point the terms with
$\operatorname{md}^{-\delta}$ for small $\delta$, we obtain
\begin{align*}
0 &\le
-(\operatorname{md}^{-\delta})'(F(x))
\left(1+\left.\frac{d}{ds}\right|_{s=0}^{-}F(\eta_s)\right)
\\
&\le \frac{1}{a}\Bigg[
-2\left(\frac{F(x)^2}{2}+\frac{G(x)^2}{2}\right)
\\
&\qquad\qquad
+\left(\frac{F(\eta_{-a})^2}{2}
      +\frac{G(\eta_{-a})^2}{2}\right)
+\left(\frac{F(\eta_a)^2}{2}
      +\frac{G(\eta_a)^2}{2}\right)
\Bigg]+\Psi.
\end{align*}
From Proposition \ref{Prop: Estimate on Busemann sum}, we know that
\[
-C\le L^2-\frac{F(z)^2}{2}-\frac{G(z)^2}{2}\le C
\qquad\text{for }z\in J(\Upsilon_{-20},\Upsilon_{20}),
\]
for some positive constant $C$ independent of $\delta,L$.
The expression in brackets is therefore bounded above by $4C$,
so the preceding right-hand side is bounded above by
$4C/a+\Psi$. Since $a$ is fixed, it follows that the difference
of right and left derivatives of
$\operatorname{md}^{-\delta}\circ F\circ\eta$ is bounded below
by $0$ and above by a constant independent of $\delta,L$. Thus, dividing by the large and growing term $(\mathrm{md}^{-\delta})'(F(x)) = F(x) + \Psi$, we obtain
    \begin{equation*}
        1 \leq {-} \frac{d}{ds}\bigg|_{s = 0}^-F(\eta_s) \leq 1 + \Psi.
    \end{equation*}
From here suppose that the properties $i)$ and $ii)$ of Theorem \ref{Theorem: local existence} hold at $t=0$. This assumption is allowed because it holds almost every time, and we will use it just to obtain an upper bound for the right derivative of $F(\overline{\gamma}_t)$, that we will integrate to obtain the conclusion. Let $v:= \dot \gamma_0$, hence $v^{\flat} \in \partial_\leq \, \mathrm{md}^{-\delta} \circ F(x)$. The defining properties of the gradient flow $\gamma_t$ and its reparameterization $\bar \gamma_t$, the fact that the geodesic $\eta_s$ is past directed and unit speed, and the reverse Cauchy--Schwarz inequality yields
\begin{equation*}
    (\mathrm{md}^{-\delta})' (F(x)) ( 1+ \Psi) \geq - \frac{d}{ds}\bigg|_{s = 0}^-\mathrm{md}^{-\delta} \circ F (\eta_s) \geq v^{\flat} [-\dot \eta_0] \geq \| v\|_{\mathsf g} \, ,
\end{equation*}
and, moreover, we have
\[
\| v\|_{\mathsf g} = (\mathrm{md}^{-\delta})' (F(x)) \, \sqrt{\frac{d}{dt}\bigg|_{t = 0}^+ F(\bar \gamma_t)} \, .
\]
Hence, we get
    \begin{equation*}
        1 + \Psi \geq - \frac{d}{ds}\bigg|_{s = 0}^-F(\eta_s) \geq  \sqrt{\frac{d}{dt}\bigg|_{t = 0}^+ F(\bar \gamma_t)}.
    \end{equation*}
    By arbitrariness of the point and uniqueness of gradient flows, these estimates hold {uniformly} for almost any $t \in [0,4]$, and the claim follows from the fundamental theorem of calculus.
\end{proof}

\begin{Definition}[The future and past flow maps]
Let $(\mathcal M,\mathsf g)$ be globally hyperbolic with timelike sectional
curvature bounded below by $-\delta\leq0$, and let
$\Upsilon:[-L,L]\to \mathcal M$ be a maximizing unit speed timelike
geodesic. Given $x\in H $ (where $H$ is as in \eqref{eq: H}), define
$t\mapsto\Fl^{+}_t(x)$ to be the reparametrized flow
$\bar\gamma_t=\gamma_{g(t)}$ of the function
$F=\ell(\Upsilon_{-L},\cdot)$, which in turn arises from the
gradient flow of the causally convex function
$\operatorname{md}^{-\delta}\circ F$ starting at $x$ and
maximal in $H$.
Similarly, we define $t\mapsto\Fl^{-}_t(x)$ as the
past-directed reparametrized flow of
$G=\ell(\cdot,\Upsilon_L)$, arising from the gradient flow of
$\operatorname{md}^{-\delta}\circ G$ maximal in $H$.
Both flow maps are defined on the respective maximal time
intervals of their reparametrizations.
\end{Definition}

Assuming $\delta \ll L^{-4} \ll 1$ as usual, it follows from Lemma \ref{Lemma: Almost monotonicity of reparam grad flow} (and its analogue for $G$) that $\{\Fl^\pm_t(x)\}_{t \in [0,4]}$ is well-defined for all $x \in I(\Upsilon_{-10},\Upsilon_{10})$ and stays in $I(\Upsilon_{-15},\Upsilon_{15})$.

\begin{Proposition}[Almost Pythagoras formula]
\label{Prop: almost pythagoras flow formula}
Let $x,y \in J(\Upsilon_{-1},\Upsilon_1)$ such that $x \leq y$ and let $s,t \in [0,2]$ be such that $\Fl^+_s(x) \leq \Fl^+_t(y)$. Then
\begin{equation}\label{future Pythagoras}
\ell(\Fl^+_s(x),\Fl^+_t(y))^2 = \ell(x,y)^2 + (t-s)^2 + 2(t-s) (F(y) - F(x)) + \Psi(\delta,L).
\end{equation}
If instead $x \geq y$ and $s,t \in [0,2]$ such that $\Fl^-_s(x) \geq \Fl^-_t(y)$, then
\begin{equation}\label{past Pythagoras}
    \ell(\Fl^-_t(y),\Fl^-_s(x))^2 = \ell(y,x)^2 + (t-s)^2 + 2(t-s) (G(y) - G(x)) + \Psi(\delta,L).
\end{equation}
\end{Proposition}
\begin{proof}
    We obtain the formula involving $\Fl^+$ below. The derivation of the formula involving $\Fl^-$ is analogous. In Steps 1 and 2, we prove the assertions for initial points
$x,y\in J(\Upsilon_{-4},\Upsilon_4)$, with the same time bounds
and causal assumptions. This permits their application to the
flowed points occurring in Step 3.

    \begin{enumerate}
        \item \textbf{STEP 1: $t = 0$.} The inequality $\geq$ is a consequence of Proposition \ref{Prop: second estimate for reparametrized flow}. Consider now $[0,4] \ni t \mapsto \Fl^-_t(\Fl^+_s(x))$. Using the time-reversed one-sided case of Proposition \ref{Prop: second estimate for reparametrized flow} for $\ell(\Fl^-_t(\Fl^+_s(x)),y)^2$, in conjunction with the above noted inequality $\geq$ in our claim,
        we obtain
        \begin{align*}
            \ell(\Fl^-_t(\Fl^+_s(x)),y)^2 &\geq \ell(\Fl^+_s(x),y)^2 + t^2 + 2t(G(\Fl^+_s(x)) - G(y)) + \Psi\\
            &\geq \ell(x,y)^2 + s^2 - 2s(F(y) - F(x)) 
            + t^2 + 2t(G(\Fl^+_s(x)) - G(y)) + \Psi.
        \end{align*}
        From $F = {2L -}G + \Psi$ (Proposition \ref{Prop: Estimate on Busemann sum}) as well as the almost linearity provided by Proposition \ref{Proposition: Almost linearity},
        we continue estimating 
        \begin{align}\nonumber
            \ell(\Fl^-_t(\Fl^+_s(x)),y)^2 &\geq \ell(x,y)^2 + s^2 -2s(F(y) - F(x)) 
            +t^2 +2t(-F(x) - s + F(y)) + \Psi\\
            &= \ell(x,y)^2 + (t-s)^2 + 2(t-s)(F(y) - F(x)) + \Psi.
        \label{3.41bis}\end{align}
        Similarly, we may estimate for $t > 0$ and $x \leq \Fl^-_t(\Fl^+_s(x))$ that
        \begin{align}
            \ell(x,\Fl^-_t(\Fl^+_s(x)))^2 \geq (t-s)^2 + \Psi. \label{eq3.41}
        \end{align}
        The case $s=0$ is immediate, so suppose $s>0$ and put
$q:=\Fl^+_s(x)$. By $1$-steepness of $F$ and the
lower bound $1$ for the Lorentzian speed of the past flow,
\begin{equation*}
F(\Fl^-_r(q))\le F(q)-r.
\end{equation*}
Consequently, this curve cannot remain in $I^+(x)$ beyond
$r=F(q)-F(x)=s+\Psi<3$. Since $q\in J(x,y)$, its existence
up to this time follows from Lemma \ref{Lemma: Almost monotonicity of reparam grad flow}. Thus there is a unique
first contact time $\bar t\in[0,3]$ such that
$\Fl^-_{\bar t}(q)\in\partial J^+(x)$ and
$\ell(x,\Fl^-_{\bar t}(q))=0$.
From the above estimate \eqref{eq3.41}, we obtain that
        \begin{align*}
            0 \geq (\bar t - s)^2 + \Psi \geq \Psi
        \end{align*}
        Subtracting $\Psi$ from both sides and taking the square root we get
        \begin{align*}
            \bar t = s + \Psi.
        \end{align*}
        Thus, the reverse triangle inequality in conjunction with our estimate \eqref{3.41bis} produces
        \begin{align*}
            \ell(x,y)^2 &\geq  (\underbrace{\ell(x,\Fl^-_{\bar t}(\Fl^+_s(x)))}_{=0} + \ell(\Fl^-_{\bar t}(\Fl^+_s(x)), y))^2 \\
            &\geq \ell(x,y)^2 + (\bar t -s)^2 + 2(\bar t - s) (F(y) - F(x)) + \Psi.
        \end{align*}
        Given that $\bar t=s+\Psi$ and $|F(y)-F(x)|\le8$,
the right-hand side is $\ell(x,y)^2+\Psi$. Thus, all of the inequalities have to be equalities, so that in particular
        \begin{align*}
        \ell(\Fl^+_s(x),y)^2 = \ell(x,y)^2 + s^2 - 2s(F(y) - F(x)) + \Psi.
        \end{align*}
        \item \textbf{STEP 2: $s = 0$.} Let $y_r:=\mathrm{Fl}^+_r(y)$, and first assume $x\ll y$.
By Lemma \ref{Lemma: Almost monotonicity of reparam grad flow}, the points under consideration lie in
$I(\Upsilon_{-15},\Upsilon_{15})$.
During the course of this step only, we write
\begin{equation*}
 U:=\operatorname{md}^{-\delta}\circ F,
 \qquad
 V:=\operatorname{md}^{-\delta}\circ G,
 \qquad
 D(q):=(\operatorname{md}^{-\delta})'(F(q)).
\end{equation*}
By Proposition \ref{Prop: Estimate on Busemann sum},
\begin{equation*}
 |U+V-L^2|\leq C_0
 \quad\text{on }J(\Upsilon_{-20},\Upsilon_{20}),
\end{equation*}
where $C_0$ is independent of $\delta,L$, and
\begin{equation*}
 U=\frac{F^2}{2}+\Psi,\qquad
 V=\frac{G^2}{2}+\Psi,\qquad
 D=F+\Psi,\qquad F+G=2L+\Psi.
\end{equation*}
We first establish an endpoint estimate.
Fix $x\ll q$ in $I(\Upsilon_{-15},\Upsilon_{15})$ and an
affinely parametrized maximizing timelike geodesic
$\sigma:[0,1]\to M$ from $x$ to $q$.
We claim that its (not necessarily maximizing) geodesic extension is defined on $[0,1+h]$, where
$h:=1/1000$, and remains in
$I(\Upsilon_{-20},\Upsilon_{20})$. This is shown via a convexity argument, cf.\ the proof of Proposition \ref{Proposition: Almost linearity}.

Put $w:=\dot\sigma_1$ and $z:=\sigma_{1+h}$, and suppose $u \in T_q \mathcal M$ satisfies
$D(q)u^\flat\in\partial_{\le}U(q)$ (this will subsequently be applied to $u = \dot y_r$).
Testing the causal subdifferential inequality with $-w$ and $hw$
gives
\begin{equation*}
 U(q)-U(x)
 \le D(q) \mathsf g(u,w)
 \le\frac{U(z)-U(q)}h.
\end{equation*}
Convexity of $V\circ\sigma$ gives
\begin{equation*}
 V(z)-V(q)\ge h\bigl(V(q)-V(x)\bigr).
\end{equation*}
Together with the bound on $U+V$, these inequalities imply
\begin{equation*}
 \frac{U(q)-U(x)}{D(q)}
 \le \mathsf g(u,w)
 \le
 \frac{V(x)-V(q)}{D(q)}+\frac{2C_0}{hD(q)}.
\end{equation*}
Since $F,G\in[L-15,L+15]$ at $x,q$, the preceding
asymptotic estimates yield
\begin{equation*}
 \frac{U(q)-U(x)}{D(q)}=F(q)-F(x)+\Psi,
 \qquad
 \frac{V(x)-V(q)}{D(q)}=F(q)-F(x)+\Psi.
\end{equation*}
Moreover, $2C_0/(hD(q))=O(L^{-1})=\Psi$.
We have therefore proved that
\begin{equation*}
 \mathsf g(u,\dot\sigma_1)=F(q)-F(x)+\Psi.
\end{equation*}
For almost every $r\in(0,t)$, $y_r$ is differentiable
and satisfies
\begin{equation*}
 D(y_r)\dot y_r^\flat\in\partial_{\le}U(y_r).
\end{equation*}
Fix such an $r$, and let $\sigma:[0,1]\to M$ be an affinely
parametrized maximizing geodesic from $x$ to $y_r$.
Lemma \ref{L:gradient estimate} gives
\begin{equation*}
 \left.\frac{d^+}{ds}\right|_{s=r}
 \frac{\ell(x,y_s)^2}{2}
 \ge \mathsf g(\dot y_r,\dot\sigma_1).
\end{equation*}
For the corresponding left Dini derivative, choose $\rho<1$
sufficiently close to $1$ so that
$p\mapsto\ell(\sigma_\rho,p)$ is smooth near $y_r$.
For sufficiently small $\varepsilon>0$, the reverse triangle
inequality gives
\begin{equation*}
 \ell(x,y_{r-\varepsilon})
 \ge
 \ell(x,\sigma_\rho)+\ell(\sigma_\rho,y_{r-\varepsilon}),
\end{equation*}
whereas maximality of $\sigma$ gives
\begin{equation*}
 \ell(x,y_r)
 =
 \ell(x,\sigma_\rho)+\ell(\sigma_\rho,y_r).
\end{equation*}
Squaring, subtracting the former inequality from the latter
equality, and dividing by $2\varepsilon$, the usual first
variation formula yields
\begin{equation*}
 \left.\frac{d^-}{ds}\right|_{s=r}
 \frac{\ell(x,y_s)^2}{2}
 \le \mathsf g(\dot y_r,\dot\sigma_1).
\end{equation*}
Since $x\ll y$, the function
$r\mapsto\ell(x,y_r)^2/2$ is locally Lipschitz.
Consequently, for almost every $r$, its left and right Dini derivatives both agree with its ordinary derivative.
From the inequalities established above and
Proposition \ref{Proposition: Almost linearity} we can thus conclude that
\begin{equation*}
 \frac{d}{dr}\frac{\ell(x,y_r)^2}{2}
 =
 g(\dot y_r,\dot\sigma_1)
 =
 F(y_r)-F(x)+\Psi
 =
 F(y)-F(x)+r+\Psi
\end{equation*}
for almost every $r\in(0,t)$. Integrating from $0$ to $t$, we get
\begin{equation*}
 \ell(x,\mathrm{Fl}^+_t(y))^2
 =
 \ell(x,y)^2+t^2+2t\bigl(F(y)-F(x)\bigr)+\Psi.
\end{equation*}
The case $x\leq y$ but $x \not \ll y$ is argued by approximation.
        
        \item \textbf{STEP 3: General case.} If both $s, t \neq 0$, then one can simply proceed in two steps using steps 1 and 2 as follows:
        \begin{align*}
            \ell(\Fl^+_s(x),\Fl^+_t(y))^2 &= \ell(x,\Fl^+_t(y))^2 + s^2 - 2s(\underbrace{F(\Fl^+_t(y))}_{= F(y) + t + \Psi} - F(x)) + \Psi\\
            &= \ell(x,y)^2 + t^2 + s^2 - 2s(F(y) + t  - F(x)) + 2t(F(y) - F(x)) + \Psi\\
            &= \ell(x,y)^2 + (t-s)^2 + 2(t-s)(F(y) - F(x)) + \Psi.
        \end{align*}
    \end{enumerate}
\end{proof}

By rescaling, the same formulas \eqref{future Pythagoras} and \eqref{past Pythagoras} hold for every fixed $R>0$,
with initial points $x,y\in J(\Upsilon_{-R},\Upsilon_R)$
and times $s,t\in[0,2R]$, under the same causal assumptions
and provided the relevant flow segments are defined in $H$.
The error then has the form $\Psi(\delta,L\mid R)$. This will be used in a technical step in the proof of the main theorem.

\begin{Proposition}[Flow metrization of a spacelike slice]
\label{Prop: flow metrization}
Let $X:=J(\Upsilon_{-1},\Upsilon_1) \cap \{b_L^+ = 0\}$ with $b_L^+$ from \eqref{b^+_L} and for $x,y \in X$ define
\begin{equation}
\label{eq: def tilde d}
    \tilde {\mathsf{d}}(x,y):= \min\{t \in [0,4] : \Fl^+_t(x) \geq y \text{ and } \Fl^+_t(y) \geq x\}.
\end{equation}
Then there exists $0 \leq \theta \leq  \Psi$ such that the {infimum} in the definition of $\tilde {\mathsf{d}}$ can be taken on $[0,2+\theta]$. Moreover, defining $\mathsf d:=\tilde{\mathsf {d}} + \theta$ on $X^2 \setminus \{(x,x) : x \in X\}$ and $\mathsf{d}(x,x):=\tilde{\mathsf{d}}(x,x) = 0$, the pair $(X,\mathsf{d})$ is a metric space. Finally, if $x,y \in X$, then
\begin{equation*}
    \ell(x,\Fl^+_{\tilde{\mathsf{d}}(x,y)}(y)) = 0 \quad \text{or} \quad \ell(y,\Fl^+_{\tilde{\mathsf{d}}(x,y)}(x)) = 0.
\end{equation*}
\end{Proposition}
\begin{proof}
    Fix $x,y \in X$. Since $y \in J(\Upsilon_{-1},\Upsilon_1)$, by Lemma \ref{Lemma: Almost monotonicity of reparam grad flow}, $\Fl^+_t(y) \in J(\Upsilon_{-1 + t + \Psi}, \Upsilon_{1 + t + \Psi})$. By a continuity argument we may obtain a $\theta \geq 0$ of type $\Psi$ such that, if $t \geq 2 + \theta$, then $\Fl^+_t(y) \geq x$ as well as $\Fl^+_t(x) \geq y$.

    Let us show that $\mathsf d$ is a metric on $X$. Clearly, $\tilde {\mathsf{d}}(x,x) = \mathsf{d}(x,x) = 0$. Symmetry of $\tilde {\mathsf{d}}$ (and thus of $\mathsf{d}$) follows from the definition. It remains to argue the triangle inequality for $\mathsf{d}$, for which it is sufficient to show that the triangle inequality holds for $\tilde{\mathsf{d}}$ up to an addition of $\theta$. To this end, let $x,y,z \in X$, define $T:= \tilde{\mathsf{d}}(x,y)$, $S:=\tilde{\mathsf{d}}(y,z)$. By definition,
    \begin{equation*}
        \Fl^+_T(x) \geq y, \, \Fl^+_T(y) \geq x, \, \Fl^+_S(z) \geq y, \, \Fl^+_S(y) \geq z.
    \end{equation*}
    For sufficiently large $\xi$, certainly $\Fl^+_{S+ \xi}(z) \geq \Fl^+_T(y)$. Using Proposition \ref{Prop: second estimate for reparametrized flow}, after a fixed rescaling if necessary,
and noting that $F(z)=L=F(y)$, we obtain  
    \begin{align*}
        \ell(\Fl^+_{T}(y),\Fl^+_{S + \xi}(z))^2 &\geq \underbrace{\ell(y,\Fl^+_S(z))^2}_{ \ge 0} + (\xi - T)^2 + 2(\xi - T)(\underbrace{F(\Fl^+_S(z))}_{=S + \Psi} - F(y)) + \Psi \\
        &{\ge} (\xi - T)^2 + 2(\xi - T)S + \Psi.
    \end{align*}
    Let $\xi_0$ be the first nonnegative parameter for which
$\Fl^+_{S+\xi_0}(z)\ge\Fl^+_T(y)$.
If $\xi_0>T+\theta$, the left-hand side of the above estimate
vanishes at $\xi_0$, whereas its right-hand side is strictly
positive after increasing $\theta$ within the class $\Psi$,
a contradiction. Thus $\xi_0\leq T+\theta$, hence
$x\leq\Fl^+_T(y)\leq\Fl^+_{T+S+\theta}(z)$. Analogously, $z \leq \Fl^+_{T + S + \theta} (x)$. Thus,
    \begin{equation*}
        \tilde{\mathsf{d}}(x,z) \leq T + S + \theta = \tilde{\mathsf{d}}(x,y) + \tilde{\mathsf{d}}(y,z) + \theta.
    \end{equation*}
    To show the final claim, let $T:=\tilde{\mathsf{d}}(x,y)$ and suppose that 
    \begin{equation*}
        \ell(x,\Fl^+_T(y)) > 0 \quad \text{and} \quad \ell(y,\Fl^+_T(x)) > 0.
    \end{equation*}
    By continuity, there is $\varepsilon > 0$ such that $x \leq \Fl^+_{T - \varepsilon}(y)$ and $y \leq \Fl^+_{T - \varepsilon}(x)$, which would imply $T\leq T-\varepsilon$, a contradiction.
\end{proof}

We will also use the larger spacelike slice
\begin{equation}
\label{eq: extended slice}
\widehat X:=J(\Upsilon_{-5/2},\Upsilon_{5/2})\cap\{b_L^+=0\},
\end{equation}
which contains $X$. Applying Proposition \ref{Prop: flow metrization} after the
rescaling $\mathsf g\mapsto 4 \mathsf g/25$ gives the same construction on
$\widehat X$, with the interval $[0,10]$ in \eqref{eq: def tilde d} and with
$\tilde{\mathsf d}\leq 5+\theta$. The required flow segments are
defined in $H$ by successive applications of Lemma \ref{Lemma: Almost monotonicity of reparam grad flow}.
Choosing a common $\theta$ of type $\Psi$ for both constructions,
we obtain a metric $\mathsf d$ on $\widehat X$ whose restriction to $X$
is the metric above. Indeed, for points of $X$ the minimum in
\eqref{eq: def tilde d} is already attained in $[0,2+\theta]$.
We use the same notation $\tilde{\mathsf d}$ and $\mathsf d$ for these
extensions.

\begin{Proposition}[Almost symmetry]
\label{P:almost symmetry}
Let $x,y \in X$ and suppose that $\ell(x,\emph{\Fl}^+_{\tilde{\mathsf{d}}(x,y)} (y)) =  0$, then
\begin{equation}
    0 \leq \ell(y,\emph{\Fl}^+_{\tilde{\mathsf{d}}(x,y)}(x)) = \Psi.
\end{equation}
In particular, also $\ell(y,\Fl^+_{\mathsf{d}(x,y)}(x)) =\Psi$.
\end{Proposition}
\begin{proof}
    By the first-contact argument in Proposition \ref{Prop: flow metrization} and the almost-Pythagoras formula \eqref{future Pythagoras}, we have that
    \[
    0 < \ell(\Fl^+_{\tilde{\mathsf{d}}(x,y)} (x),\Fl^+_{2\tilde{\mathsf{d}}(x,y) + \Psi}(y)) = \Psi \, .
    \]
    We may assume that $0 < \ell(y,\Fl^+_{\tilde{\mathsf{d}}(x,y)}(x))$, otherwise there is nothing to show.  Notice that, since $F(x) = L$ (equivalently, $b^+_L(x) = 0$), by almost linearity of $F$ along $\Fl^+$ we have that $F( \Fl^+_{\tilde{\mathsf{d}}(x,y)}(x)) = F(x) + \tilde{\mathsf{d}}(x,y) + \Psi = L + \tilde{\mathsf{d}}(x,y) + \Psi$.
    Choose $K=K(\delta,L)\to\infty$ sufficiently slowly so that
$K = o(L)$ and
\begin{equation*}
K^2\Psi(\delta K^2,L/K)\longrightarrow0,
\end{equation*}
where $\Psi$ denotes the error in \eqref{future Pythagoras}. Connect $\Upsilon_{-K}$ to $\Fl^+_{2\tilde{\mathsf{d}}(x,y) + \Psi}(y)$ via a maximizing future directed $\mathsf g$-arclength parametrized timelike geodesic $\sigma:[0,S] \to M$, where $S:=\ell(\Upsilon_{-K},\Fl^+_{2\tilde{\mathsf{d}}(x,y) + \Psi}(y))$. Note that $\sigma$ intersects $\partial J^-(y)$ uniquely at a point $z$, and since $y \ll \Fl^+_{\tilde{\mathsf{d}}(x,y)}(x)$ by assumption, we have $z \ll \Fl^+_{\tilde{\mathsf{d}}(x,y)}(x)$. Now consider the timelike triangle $\Delta(\Upsilon_{-K},\Fl^+_{\tilde{\mathsf{d}}(x,y)}(x), \Fl^+_{2\tilde{\mathsf{d}}(x,y) + \Psi}(y))$. By triangle comparison,
    \begin{equation}
        0 < \ell(z, \Fl^+_{\tilde{\mathsf{d}}(x,y)}(x)) \leq \bar \ell_{-\delta}(\bar z, \bar x_{\tilde{\mathsf{d}}}),
    \end{equation}
    where $\bar x_{\tilde{\mathsf{d}}}$ is the point corresponding to $\Fl^+_{\tilde{\mathsf{d}}(x,y)}(x)$ on the comparison triangle in $\mathbb{L}^2(-\delta)$. 
    Set $F_K:=\ell(\Upsilon_{-K},\cdot)$ and
$G_K:=\ell(\cdot,\Upsilon_K)$. Proposition \ref{Prop: Estimate on Busemann sum} applied to
$\Upsilon|_{[-K,K]}$, together with the reverse triangle
inequality, gives on $J(\Upsilon_{-20},\Upsilon_{20})$
\begin{equation*}
0\le F-(L-K+F_K)\le2K-F_K-G_K\le\Psi(\delta,K).
\end{equation*}
In particular,
    \begin{equation} \label{eq: hinge}
        F( \xi) = F \left (\Upsilon_{-K} \right ) + \ell ( \Upsilon_{-K}, \xi) + \Psi \, .
    \end{equation}
    Consider the rescaled metric $\mathsf g'=K^{-2}\mathsf g$. Its curvature parameter is $\delta K^2$ and the rescaled half-length of $\Upsilon$ is $L/K$. Since
$\Upsilon_{-K}\le z\le y\le\Upsilon_K$, we may apply \eqref{future Pythagoras}
to $z,y$ in the rescaled metric. Rescaling back, its error
is $K^2\Psi(\delta K^2,L/K)=\Psi$, and we obtain
    \begin{align} \label{eq: quadr}
        \ell(z, \Fl^+_{2\tilde{\mathsf{d}}(x,y) + \Psi}(y))^2 &= (2\tilde{\mathsf{d}}(x,y) + \Psi)^2 + 2(2\tilde{\mathsf{d}}(x,y) + \Psi)(L-F(z)) + \Psi \, .
    \end{align}
    Put $T:=2\tilde{\mathsf d}(x,y)+\Psi$,
$p:=\Fl^+_T(y)$ and $c:=\ell(z,p)$.
Since $z$ lies on the maximizing geodesic from
$\Upsilon_{-K}$ to $p$, the reverse triangle inequality
and \eqref{eq: hinge} at $p$ give
\begin{equation*}
0\le F(p)-F(z)-c
\le F(p)-(L-K+F_K(p))=\Psi.
\end{equation*}
Thus $F(z)+c=L+T+\Psi$, and \eqref{eq: quadr} gives
$(c-T)^2=\Psi$, which establishes that
    \begin{equation*}
        0 < \ell(z, \Fl^+_{2\tilde{\mathsf{d}}(x,y) + \Psi}(y)) = 2\tilde{\mathsf{d}}(x,y) + \Psi.
    \end{equation*}
    We now apply the law of cosines to the comparison triangle
of
\begin{equation*}
\Delta\bigl(\Upsilon_{-K},
\Fl^+_{\tilde{\mathsf d}(x,y)}(x),p\bigr),
\end{equation*}
where $\bar z$ lies on the side corresponding to the
maximizer from $\Upsilon_{-K}$ to $p$.
The preceding estimates give
\begin{align*}
\ell(\Upsilon_{-K},z)&=K+\Psi,\\
\ell(z,p)&=2\tilde{\mathsf d}(x,y)+\Psi,\\
\ell(\Upsilon_{-K},
\Fl^+_{\tilde{\mathsf{d}}(x,y)}(x))
&=K+\tilde {\mathsf{d}}(x,y)+\Psi,\\
\ell(\Fl^+_{\tilde{\mathsf  d}(x,y)}(x),p)&=\Psi.
\end{align*}
Thus the asymptotic formula in the law of cosines, whose remainder is
$O(\delta K^4)=\Psi$, yields
    \begin{align*}
        \ell_{-\delta}(\bar z, \bar x_{\tilde{\mathsf{d}}})^2 = \frac{(K +\tilde{\mathsf{d}}(x,y) + \Psi)^2(2\tilde{\mathsf{d}}(x,y) + \Psi) + \Psi^2 (K + \Psi)}{K + 2\tilde{\mathsf{d}}(x,y) + \Psi} - (K+\Psi)(2 \tilde{\mathsf{d}}(x,y) + \Psi) + \Psi.
    \end{align*}
    Upon bringing the terms to the same denominator, it is easily checked that $\bar \ell_{-\delta}(\bar z, \bar x_{\tilde{\mathsf{d}}}) = \Psi$, which on the other hand lets us conclude that
    \begin{align*}
        0 < \ell(z, \Fl^+_{\tilde{\mathsf{d}}(x,y)}(x)) = \Psi.
    \end{align*}
    Observing that $\ell(z, \Fl^+_{\tilde{\mathsf{d}}(x,y)}(x)) \geq \ell(y, \Fl^+_{\tilde{\mathsf{d}}(x,y)}(x)) > 0$ by the reverse triangle inequality concludes the proof.
\end{proof}

\begin{Remark}[{Past and future flows approximately invert each other}]
\label{Remark: flows are inverse}
In Proposition \ref{Prop: almost pythagoras flow formula}, we obtain as a consequence of its proof that if $x \in J(\Upsilon_{-1},\Upsilon_1)$ and $s \in [0,2]$, then there exists a unique $\bar t = s + \Psi$ such that $\Fl^-_{\bar t}(\Fl^+_s(x)) \in \partial J^+(x)$, and similarly with $\Fl^+$ and $\Fl^-$ interchanged.
Also note that if $y\in J(\Upsilon_{-1},\Upsilon_1)$, then
\begin{equation*}
\ell_+(x,y)
=\ell_+\bigl(\Fl^-_{\bar t}
(\Fl^+_s(x)),y\bigr)+\Psi.
\end{equation*}
This follows from \eqref{future Pythagoras}-\eqref{past Pythagoras}, first flowing $y$ a fixed
amount to the future and then comparing the resulting
formulas up to the first causal contact.
\end{Remark}

\begin{Corollary}
\label{Corollary: mixed flows}
    Let $x,y \in J(\Upsilon_{-1},\Upsilon_1)$ and $s,t \in [0,2]$ be such that $\Fl^-_s(x) \leq \Fl^+_t(y)$. Then
    \begin{equation}\label{mixed Pythagoras}
        \ell(\Fl^-_s(x),\Fl^+_t(y))^2 = \ell_+(x,y)^2+(t+s-d)^2
+2(t+s-d)(b_L^+(x)-b_L^+(y)+d)+\Psi
    \end{equation}
    where $d \geq 0$ is the minimal time such that $\Fl^+_d(y) \geq x$. 
    In particular, if $x,y \in X$, then
    \begin{equation}
    \label{eq: pythagoras from slice}
        \ell(\Fl^-_s(x),\Fl^+_t(y))^2 = (t+s)^2 - \mathsf{d}(x,y)^2 + \Psi.
    \end{equation}
\end{Corollary}

\begin{proof}
Let $d$ be the minimal time such that
$\Fl^+_d(y)\ge x$, and set $q:=\Fl^+_d(y)$.
Then $\ell(x,q)^2=\ell_+(x,y)^2$.
Since $\Fl^-_s(x)\leq x\leq q$, the past almost
Pythagoras formula \eqref{past Pythagoras}, together with almost linearity \eqref{linearity},
gives
\begin{equation*}
\ell(\Fl^-_s(x),q)^2
=\ell_+(x,y)^2+s^2
+2s\bigl(b_L^+(x)-b_L^+(y)+d\bigr)+\Psi.
\end{equation*}
If $t\geq d$, we apply \eqref{future Pythagoras} to
$\Fl^-_s(x)\le q$, flowing $q$ forward for time
$t-d$. If $t<d$, we instead apply \eqref{future Pythagoras} to $\Fl^-_s(x)\leq\Fl^+_t(y)$, flowing the
second point forward for time $d-t$, and rearrange.
In both cases, almost linearity \eqref{linearity} yields
\begin{align*}
\ell(\Fl^-_s(x),\Fl^+_t(y))^2
&=\ell(\Fl^-_s(x),q)^2+(t-d)^2\\
&\quad+2(t-d)\bigl(b_L^+(x)-b_L^+(y)+d+s\bigr)+\Psi\\
&=\ell_+(x,y)^2+(t+s-d)^2\\
&\quad+2(t+s-d)\bigl(b_L^+(x)-b_L^+(y)+d\bigr)+\Psi,
\end{align*}
which proves \eqref{mixed Pythagoras}.

For $x,y\in X$, we have $b_L^+(x)=b_L^+(y)=0$ and
$\ell_+(x,y)=0$. Proposition \ref{P:almost symmetry} and \eqref{future Pythagoras} give
$\tilde{\mathsf d}(x,y)^2-d^2=\Psi$.
Since $\tilde{\mathsf d}(x,y)\ge d$, it follows that
$d=\tilde{\mathsf d}(x,y)+\Psi=\mathsf d(x,y)+\Psi$,
which proves \eqref{eq: pythagoras from slice}.
\end{proof}

By fixed rescaling, Proposition \ref{P:almost symmetry} and \eqref{eq: pythagoras from slice} also hold for $x,y\in\widehat X$, where $\hat X$ is the extended spacelike slice from \eqref{eq: extended slice}.

\subsection{Proof of the main result}
\label{Subsection: Proof of main result}

We are now in position to prove our main result. To this end, recall the flow metrized spacelike slice $(X,\mathsf{d})$
from Proposition \ref{Prop: flow metrization}, where $X:=  J^{\mathcal M}(\Upsilon_{-1},\Upsilon_1)\cap\{b^+_L=0\}$ (we write $J^\mathcal M$ in this Subsection only to distinguish diamonds in $\mathcal M$ from those in $\R \times X$). If
$X\cap\partial J^{\mathcal M}(\Upsilon_{-1},\Upsilon_1)\neq\emptyset$, choose $w$ in this intersection. Otherwise, if possible, choose $w\in X$ with
$\mathsf d(\Upsilon_0,w)>1$. In either case, \eqref{eq: pythagoras from slice} gives
$\mathsf d(\Upsilon_0,w)=1+\Psi$, and we replace the metric on
$\widehat X$ by
\begin{equation*}
\frac{\mathsf d}{\mathsf d(\Upsilon_0,w)},
\end{equation*}
retaining the notation $\mathsf d$. This preserves all preceding
estimates up to $\Psi$ and ensures that the product spacelike
boundary is nonempty. If neither choice is possible, we keep the original metric.
In this case no assignment to the product spacelike boundary
will be needed below: a level projection leaving $J^{\mathcal M}(\Upsilon_{-1},\Upsilon_1)$, or a point
of $i^0$, would produce a point of $X\cap\partial J^{\mathcal M}(\Upsilon_{-1},\Upsilon_1)$ by intersecting
an appropriate null boundary generator with $\{b_L^+ = 0\}$.

We will consider the corresponding Lorentzian product $(\R \times X, \ell_{\R \times X})$ from \eqref{Lorentzian products}. Let $x_0:=\Upsilon_0 \in X$ and define the corresponding diamond
\begin{equation}\label{comparison diamond}
    J^{\R \times X}((-1,x_0),(1,x_0)) = \{(t,x) \in \R \times X:  
    \mathsf{d}(x_0,x) \leq 1 - |t|\}.
\end{equation}
Its spacelike boundary is denoted
$\bar i^0 :=$ $\{(t,x) \in \R \times X: t = 0, \mathsf d(x_0,x) = 1\}$ as in Remark \ref{Remark: i^0}.

Recall Remark \ref{Remark: flows are inverse}, which quantifies how the flows $\Fl^+$ and $\Fl^-$ are approximately inverse to each other. This will be used below.

To prove our claim, we need to define a correspondence $\cR$ between $J^\mathcal M(\Upsilon_{-1},\Upsilon_1)$ and $J^{\R \times X}((-1,x_0),(1,x_0))$ whose distortion is $\leq \Psi$. Throughout these final arguments, we understand the spacelike boundaries $i^0$  and $\bar i^0$ in the two diamonds each to be collapsed to a point 
(if nonempty). All level projections below occur by flow time $1$, by the
steepness of $F$, and belong to $\widehat X$ by Lemma \ref{Lemma: Almost monotonicity of reparam grad flow}. We define $\cR$ as follows:
\begin{enumerate}
    \item If $x \in J^\mathcal M(\Upsilon_{-1},\Upsilon_1)$ and $-b_L^+(x) \geq 0$ in \eqref{b^+_L}, the flow $t \mapsto \Fl^-_t(x)$ will intersect $\{b^+_L = 0\}$ uniquely at a time $\bar t = -b^+_L(x) + \Psi \geq 0$
    (by approximate affineness, Proposition \ref{Proposition: Almost linearity}). Let $z = \Fl^-_{\bar t}(x)$ be the intersection point.  
    If $z \in J^\mathcal M(\Upsilon_{-1},\Upsilon_1)$ and $\mathsf d(x_0,z) \leq 1$, take $(x,(\min  \{ \bar t, 1 - \mathsf d(x_0,z)\}, z)) \in \cR$,  which \eqref{comparison diamond} shows to be in the product of the diamonds to be compared. 
    Otherwise, we take $(x,{\bar i^0}) \in \cR$.

    \item If $x \in J^\mathcal M(\Upsilon_{-1},\Upsilon_1)$ and $b^+_L(x) > 0$, we proceed similarly as in (i): The flow $t \mapsto \Fl^+_t(x)$ will intersect $\{b_L^+ = 0\}$ uniquely in some $z$ at a unique time $\bar t = b_L^+(x) + \Psi \ge 0$. If $z \in J(\Upsilon_{-1},\Upsilon_1)$ and $\mathsf d(x_0,z) \leq 1$, set $(x, (-\min\{\bar t, 1 - \mathsf d(x_0,z)\}, z)) \in \cR$. Otherwise, take $(x,{\bar i^0}) \in \cR$.
 
    \item Now suppose $(t,z) \in J^{\R \times X}((-1,x_0),(1,x_0))$. Then $z \in X$. If $t \geq 0$ and $\Fl^+_t(z) \in J(\Upsilon_{-1},\Upsilon_1)$, we take $(\Fl^+_t(z), (t,z)) \in \cR$. If $\Fl^+_t(z) \notin J^\mathcal M(\Upsilon_{-1},\Upsilon_1)$, let $x \in \partial J^\mathcal M(\Upsilon_{-1},\Upsilon_1)$ be the exit point from $J^{\mathcal M}(\Upsilon_{-1},\Upsilon_1)$ of the curve
$s\in[0,t]\mapsto\Fl^+_s(z)$. In this case, set $(x,(t,z)) \in \cR$.

    \item If $(t,z) \in J^{\R \times X}((-1,x_0),(1,x_0))$ and $t < 0$, if $\Fl^-_{-t}(z) \in J(\Upsilon_{-1}, \Upsilon_1)$, we take the pair $(\Fl^-_{-t}(z), (t,z)) \in \cR$. If not, let $x$ be the exit point from
$J^{\mathcal M}(\Upsilon_{-1},\Upsilon_1)$ of the curve
$s\in[0,-t]\mapsto\Fl^-_s(z)$,
and take $(x,(t,z))\in \cR$.

    \item To complete the definition of $\cR$, 
    {if $J^{\mathcal M}(\Upsilon_{-1},\Upsilon_1)$ has a spacelike boundary $i^0 \ne \emptyset$}, we include $(i^0,{\bar i^0}) \in \cR$.
\end{enumerate}

Now we show that the distortion of $\cR$ is $\leq \Psi$, which will conclude the proof of our main result, where
\begin{equation}
    \mathrm{dis}(\cR) := \sup\{ |\ell_+(x,x') - (\ell_{\R \times X})_+((t,z),(t',z'))| : (x,(t,z)), (x',(t',z')) \in \cR\}.
\end{equation}

We proceed to justify this case by case in the following discussion. Consider a pair $(x,(t,z)),(x',(t',z')) \in \cR$ both from the graph (i)--(ii), the antigraph (iii)-(iv), or one from each.
\begin{enumerate}
\item[(vi)] If $x \not \leq x'$ and $(t,z) \not \leq (t',z')$  both $\ell(x,x')_+$ and $\ell_{\R\times X}((t,z),(t',z'))_+$ vanish so there is nothing to argue.

\item[(vii)] 

Define
\begin{equation*}
\Phi_t(z):=
\begin{cases}
\Fl^+_t(z),&t\geq0,\\
\Fl^-_{-t}(z),&t<0.
\end{cases}
\end{equation*}
For $z,z'\in X$ and $t,t'\in[-1,1]$, the almost-Pythagoras
formulas \eqref{future Pythagoras}, \eqref{past Pythagoras} and \eqref{eq: pythagoras from slice} give
\begin{equation*}
\ell_+(\Phi_t(z),\Phi_{t'}(z'))^2
=
\bigl((t'-t)_+^2-\mathsf d(z,z')^2\bigr)_++\Psi.
\end{equation*}
Indeed, for causally related endpoints this follows from
those formulas. Otherwise, increase $t'$ to the first causal
contact time $u$. At this contact,
$u-t=\mathsf d(z,z')+\Psi$, which gives the same conclusion
with positive parts.

\item[(viii)] 

We next check the adjustments to the flow times in the
definition of $\cR$. For $x\in J^\mathcal M(\Upsilon_{-1},\Upsilon_1)$ with
$b_L^+(x) \leq 0$, let $b$ be its past level-projection time, and put
$z:=\Fl^-_b(x)$ and $r:=\mathsf d(x_0,z)$. Then
$0\leq b\leq1$, and
\begin{align}
\ell(\Upsilon_{-1},x)^2&=(1+b)^2-r^2+\Psi,\\
\ell(x,\Upsilon_1)^2&=(1-b)^2-r^2+\Psi. \label{eq: ell x Upsilon1}
\end{align}
These identities follow from the almost-Pythagoras formulas \eqref{future Pythagoras}, \eqref{past Pythagoras} and \eqref{eq: pythagoras from slice}.
For the first identity, if $z\notin J(\Upsilon_{-1},\Upsilon_1)$,
first compare with $\Upsilon_{-3}$ and then move this endpoint
to $\Upsilon_{-1}$. For a future level projection, the two
signs are interchanged.

For the assignments in (i)-(ii) not involving $\bar i^0$,
these identities show that the truncation changes the
projection time only by $\Psi$. For the exit point
$x=\Fl^+_\tau(z)$ in (iii), where $\tau<t$, applying almost-Pythagoras once more gives
\begin{equation*}
0=\ell(x,\Upsilon_1)^2
=(1-\tau)^2-\mathsf d(x_0,z)^2+\Psi.
\end{equation*}
Since $\mathsf d(x_0,z)\leq1-t$, it follows that
\begin{equation*}
(t-\tau)^2
\leq(1-\tau)^2-(1-t)^2\leq\Psi.
\end{equation*}
Thus $\tau=t+\Psi$. The argument for (iv) is analogous.

By the argument of Remark \ref{Remark: flows are inverse}, flowing the comparison point
to the future or past, and the almost Pythagoras formulas,
replacing $x$ by $\Phi_t(z)$ changes $\ell_+$ by $\Psi$ in
either argument for every $(x,(t,z))\in \cR$ not involving
$\bar i^0$. Applying (vii) then gives the
required distortion bound.

\item[(ix)] 

Suppose $-b_L^+(x)\geq0$, and let $b,z,r$ be as in (viii).
If $z\in X$ but $r>1$, then \eqref{eq: ell x Upsilon1} gives
$b=\Psi$ and $r=1+\Psi$. This is the first case in which
$x$ is assigned to $\bar i^0$.

\item[(x)] 

It remains to consider the case $-b_L^+(x)\geq0$ where
$z:=\Fl^-_b(x)\notin X$, with $b$ the actual
level-projection time. Let $\omega=\Fl^-_\tau(x)$,
$0\leq\tau<b$, be the exit point of the past flow from
$J^\mathcal M(\Upsilon_{-1},\Upsilon_1)$, and set $a:=b-\tau$.
Since $\Fl^-_a(\omega)=z$, we get
\begin{align*}
0&=\ell(\Upsilon_{-1},\omega)^2
  =(1+a)^2-\mathsf d(z,x_0)^2+\Psi,\\
0&\leq\ell(x,\Upsilon_1)^2
  =(1-b)^2-\mathsf d(z,x_0)^2+\Psi.
\end{align*}
Subtracting the first equation from the second and noting that $b \leq 1$ yields
\begin{align*}
    0 \leq (1- b)^2 - (1+a)^2 + \Psi,
\end{align*}
which means that
\begin{equation*}
    2(b+a) \leq b^2 - a^2 + \Psi = (b-a)(b+a) + \Psi \leq b(b+a) + \Psi \leq b + a + \Psi. 
\end{equation*}
It follows that $b+ a \leq \Psi.$ This allows us to conclude that $b_L^+(x), b_L^+(\omega) = \Psi$. This shows that assigning $x$ to $\bar i^0$ in the product diamond is justified.
{\qed}
\end{enumerate}

\section{Outlook}
\label{section: discussion outlook}

In the setting of timelike sectional curvature bounds, it would be interesting to investigate whether our result can be improved to also consider a quantification of the maximizing property of the long segment, i.e., the case where the long segment is almost a maximizer up to a small error.

A natural next step would be to try to obtain a Lorentzian almost-splitting result under the timelike Ricci curvature bounds that are highly relevant in general relativity. A quantified almost-splitting version of Galloway's celebrated splitting theorem \cite{Galloway89splitting} for globally hyperbolic spacetimes which satisfy the strong energy condition would be very much in line with the Riemannian results of Cheeger--Colding \cite{CheegerColding96}, and could be expected to lead to important results for Ricci limit spacetimes, which are examples of spaces satisfying the Lorentzian $\mathsf {TCD}$ condition introduced in \cite{beran2024nonlinear}. As our example in Appendix \ref{section: counterexample} below shows, a direct generalization of Theorem \ref{Theorem: main theorem}, i.e., a direct Lorentzian counterpart of the Cheeger--Colding theorem, does not hold. This leads us to conjecture that an additional assumption (or possibly a replacement of the Lorentzian Gromov--Hausdorff distance with its measured counterpart, cf.\ Remark \ref{remark: unmeasured measured} below) is needed to rule out possibly richer almost-rigidity scenarios in Lorentzian signature. Based on the recent elliptic approach to Lorentzian splitting theorems \cite{QuintetI, QuintetII, CaponioOhanyanOhta}, a more detailed study and understanding of the (degenerate) elliptic $p$-d'Alembertian operator is likely to prove relevant in a study of Lorentzian almost-splitting under timelike Ricci bounds.

\appendix
\section{A Ricci flat counterexample}
\label{section: counterexample}

In this section, we show that Theorem \ref{Theorem: main theorem} does not hold if one replaces the assumption of timelike sectional curvature $\geq -\delta$ with the assumption that $\Ric(v,v) \geq -\delta g(v,v)$ for all timelike $v \in TM$, i.e., a direct analogue of Cheeger--Colding's theorem \cite{CheegerColding96} is false in Lorentzian signature.

To construct the counterexample, let $\mu > 0$. We consider exterior Schwarzschild $(\mathcal M_\mu,\mathsf g_\mu)$ of mass $\mu$, i.e., $M_\mu = \R \times (2\mu,\infty) \times S^2$, and
\begin{equation}
    \mathsf g_\mu = f_\mu dt^2 - f_\mu^{-1} dr^2 - r^2 \mathsf g_{S^2}, \quad f_\mu (r):=1 - \frac{2\mu}{r}.
\end{equation}
Here $\mathsf g_{S^2}$ is the standard round metric on $S^2$. We will show that with $\mu_j:=j^{-6}$, we get a contradiction to almost splitting as $j \to \infty$.

It is well-known that $\Ric_{\mathsf g_\mu} = 0$. Moreover, exterior Schwarzschild $(\mathcal M_\mu,\mathsf g_\mu)$ is also known to be globally hyperbolic (see e.g.\ \cite[Ex.\ 1.3.11]{WaveeqLorMfdBook}). We proceed in 4 steps to derive a contradiction to almost-splitting.\\
\noindent \textbf{STEP 1: Construction of the long timelike segments.} We fix some $\omega_* \in S^2$ and connect now $(-2j,1,\omega_*)$ to $(2j,1,\omega_*)$ via a maximizing timelike curve $c_j(t) = (t,r_j(t),\omega_*)$ (the time is strictly increasing, so it can be chosen as parameter, and the maximizer cannot have an angular component since that reduces Lorentzian length). The straight line segment between the two given points has length $4j \sqrt{1-2\mu_j}.$ Thus, since $0 < f_\mu < 1$ and $f_\mu^{-1} > 1$,
\begin{align*}
    4j \sqrt{1-2\mu_j} \leq L_{\mathsf g_{\mu_j}}(c_j)&=  \int_{-2j}^{2j} \sqrt{f_j(r_j(t)) - f_j^{-1}(r_j(t)) \dot r_j(t)^2} \, dt\\
    &\leq \int_{-2j}^{2j} \sqrt{1 - \dot{r}_j(t)^2} \, dt\\
    &\leq 4j - \frac{1}{2} \int_{-2j}^{2j} \dot r_j(t)^2 \, dt.
\end{align*}
In the last step, we used the elementary inequality $\sqrt{1-u} \leq 1 - u/2$ whenever $u \in [0,1]$. Thus,
\begin{equation*}
    \int_{-2j}^{2j} \dot r_j(\bar t)^2 \, d\bar t \leq 8j(1- \sqrt{1 - 2 \mu_j}) \leq 16 j \mu_j.   
\end{equation*}
Since $r_j(-2j) = 1$, we get
\begin{align}
    |r_j(t) - 1| &= \bigg|\int_{-2j}^t \dot r_j(\bar t) \, d\bar t\bigg|\leq \sqrt{t + 2j} \bigg(\int_{-2j}^t \dot r_j(\bar t)^2 \, d\bar t\bigg)^{1/2} \nonumber \\
    &\leq \sqrt{4j} \sqrt{16j\mu_j} = 8j \sqrt{\mu_j} \nonumber \\
    &= \frac{8}{j^2} \to 0. \label{eq: r(t) - 1}
\end{align}
Let us now parametrize the maximizer $(t,r_j(t),\omega_*)$ by proper time $s(t)$, fixing $s(0) = 0$. We have that
\begin{equation*}
    \frac{ds}{dt} = \sqrt{f_j(r_j(t)) - f_j(r_j(t))^{-1}\dot r_j(t)^2} \in (0,1].
\end{equation*}
Moreover, since $L_{\mathsf g_{\mu_j}}(c_j) \geq 4j \sqrt{1- 2 \mu_j}$, we have that
\begin{align*}
    \int_{-2j}^{2j} 1- \frac{ds}{dt} \, dt  &= 4j - L_{\mathsf g_{\mu_j}}(c_j)\\
    &\leq 4j - 4j \sqrt{1-2\mu_j}\\
    &\leq 4j (1- 1 + 2\mu_j)\\
    &= 8j \mu_j = \frac{8}{j^5}.
\end{align*}
Similarly,
\begin{equation}
\label{eq: t - s(t)}
    |t - s(t)| \leq \frac{8}{j^5} \to 0.
\end{equation}
We reparametrize $c_j$ by proper time and call $\gamma_j$ the resulting unit speed timelike maximizer restricted to the interval $[-j,j]$. This restriction is possible since $c_j$ is defined on $[-2j,2j]$ and by the previous estimate $|t - s(t)| \leq 8j^{-5}$, so that $s(-2j) < -j$ and $s(2j) > j$. Writing $\gamma_j(s) = (t_j(s),r_j(s),\omega_*)$, applying \eqref{eq: r(t) - 1} and \eqref{eq: t - s(t)} to $t=t_j(s)$ gives
\begin{equation*}
    |r_j(t_j(s))-1|\le 8j^{- 2},
    \qquad |t_j(s)-s|\le 8j^{-5}
    \qquad (|s|\le j).
\end{equation*}
Hence, in Schwarzschild coordinates,
\begin{equation*}
    \gamma_j(s)\longrightarrow (s,1,\omega_*)
\end{equation*}
locally uniformly. Thus, the timelike unit-speed maximizers $\gamma_j$ become more and more vertical in the $(t,r)$ coordinates as $j \to \infty$.\\
\noindent \textbf{STEP 2: Time separation from/to a point near the horizon.} Introduce the Regge-Wheeler tortoise coordinate
\begin{equation}\label{tortoise}
    F_\mu(r):=r + 2 \mu \log(r/2\mu -1),
\end{equation}
defined so that $F_\mu'(r) = 1/f_\mu(r)$. Setting $a:=1/2$, the strict monotonicity and surjectivity of $F_\mu:(2\mu,\infty) \to \R$ yields a unique $\rho_\mu  > 2\mu$ such that
\begin{equation*}
    F_\mu(\rho_\mu) = -a  <0.
\end{equation*}
Set 
\begin{equation}\label{Z}
    Z_\mu:=(0,\rho_\mu,\omega_*) \in \mathcal M_\mu.
\end{equation}
Since $F_\mu(4\mu) =4\mu>0$, the strict monotonicity of $F_\mu$ yields $\rho_\mu < 4\mu$. Thus, upon choosing $\mu$ small enough, we can assume that $\rho_\mu < 1$.

Consider a curve $\lambda_\mu(t):=(t,r(t),\omega_*)$, $t \in [t_0,0]$, where $r(t_0) = 1$ and $r(0) = \rho_\mu < 1$. We want $\lambda$ to be inward future null, i.e., it is a future null curve such that $r(t)$ is decreasing. Since $\dot \lambda_\mu(t) = \partial_t + r'(t) \partial_r$, $\lambda_\mu$ is null if and only if
\begin{equation*}
    0 = \mathsf g_\mu(\dot \lambda_\mu,\dot \lambda_\mu) = f_\mu(r(t)) - f_\mu(r(t))^{-1} r'(t)^2,
\end{equation*}
i.e., if and only if $r'(t) = \pm f_\mu(r(t))$. Since $F_\mu' = f_\mu^{-1}$ and we are interested in the case $r' < 0$, i.e., $r' = -f_\mu \circ r$, this is equivalent to 
\begin{equation*}
    \frac{d}{dt} F_\mu(r(t)) = -1.
\end{equation*}
Since $r(0) = \rho_\mu$ and $F_\mu(\rho_\mu) = -a$,  integration yields
\begin{equation}
\label{eq: d dt F = -1}
    F_\mu(r(t)) = -a -t.
\end{equation}
Thus, $r(t) = F_\mu^{-1}(-a-t)$ determines our inward future null curve $\lambda_\mu$. From \eqref{eq: d dt F = -1}, we get upon integrating from $t_0$ to $0$
\begin{equation*}
    t_0 = F_\mu(r(0)) - F_\mu(r(t_0)) = F_\mu(\rho_\mu) - F_\mu(1) = -a-F_\mu(1) \overset{\mu \to 0}{\longrightarrow} -a -1.
\end{equation*}
The last limit can be seen from the explicit form of $F_\mu$. Note here that $t_0$ depends on $\mu$ in general. What the calculation above shows is that the coordinate time difference $0 - t_0$ of $\lambda_\mu$ tends to $1 + a = 3/2$ as $\mu \to 0$.

Fix $T>1+a$ and $A_T:=(-T,1,\omega_*)$. 
Then $A_T \ll_{\mathsf g_\mu} \lambda_\mu(t_0) \le_{\mathsf g_\mu} \lambda_\mu(0)=Z_\mu$ where $Z_\mu := (0,\rho_\mu,\omega_*)$
from \eqref{Z}. 
We  will now deduce
\begin{equation*}
    \lim_{\mu \downarrow 0}\ell_{\mathsf g_\mu}(A_T,Z_\mu) = \sqrt{(T-a)^2 -1}.
\end{equation*}
To this end, set $b_\mu:=2\mu(1+\mu)$.
Then
\begin{equation*}
    F_\mu(b_\mu)
    =2\mu(1+\mu)+2\mu\log\mu
    \longrightarrow 0
    \qquad\text{as }\mu\downarrow 0,
    \qquad
    f_\mu(b_\mu)=\frac{\mu}{1+\mu}\le\mu.
\end{equation*}
Since $F_\mu(\rho_\mu)=-a<0$ and $F_\mu$ is strictly increasing, it
follows that $\rho_\mu<b_\mu<1$ for small enough $\mu > 0$.

Let $\sigma(t)=(t,r(t),\omega(t))$, $t\in[-T,0]$, be any future-directed
causal curve from $A_T$ to $Z_\mu$. We shall estimate the length
$L_{\mathsf g_\mu}(\sigma)$ of $\sigma$ by dividing it into two parts. Since $r(-T)=1$ and $r(0)=\rho_\mu$,
there exists a largest $t_b\in(-T,0)$ such that $r(t_b)=b_\mu$.
By causality,
\begin{equation}\label{sigma causality}
    0\le \mathsf g_\mu(\dot\sigma(t),\dot\sigma(t))
    \le f_\mu(r(t))-f_\mu(r(t))^{-1}\dot r(t)^2
\end{equation}
for every $t\in[-T,0]$. Thus $|\dot r(t)|\le f_\mu(r(t))$.
Using $F_\mu'=1/f_\mu$, we obtain
\begin{equation}
\label{eq: tb estimate}
    -t_b
    \ge \int_{t_b}^{0}\frac{|\dot r(t)|}{f_\mu(r(t))}\,dt
    \ge \left|\int_{t_b}^{0}\frac{\dot r(t)}{f_\mu(r(t))}\,dt\right|
    = F_\mu(b_\mu)-F_\mu(\rho_\mu)
    = F_\mu(b_\mu)+a.
\end{equation}
By the choice of $t_b$, we have $r(t)\le b_\mu$ for all $t\in[t_b,0]$.
Since $f_\mu (r):= 1- \frac{2\mu}r$ is increasing, $f_\mu(r(t))\le f_\mu(b_\mu)\le\mu$ on
this interval. Hence
\begin{equation}
\label{eq: final portion of sigma}
L_{\mathsf g_\mu}\bigl(\sigma|_{[t_b,0]}\bigr)
    \le \int_{t_b}^{0}\sqrt{f_\mu(r(t))}\,dt
    \le (-t_b)\sqrt{\mu}
    \le T\sqrt{\mu}.
\end{equation}
Now, we estimate the length of the remainder $\sigma|_{[-T,t_b]}$.
Since $f_\mu\le 1$ and $f_\mu^{-1}\ge 1$, we have
\begin{equation*}
g_\mu(\dot\sigma(t),\dot\sigma(t))
    \le f_\mu(r(t))-f_\mu(r(t))^{-1}\dot r(t)^2
    \le 1-\dot r(t)^2
\end{equation*}
for every $t\in[-T,t_b]$. Hence, by Cauchy--Schwarz,
\begin{align*}
    L_{\mathsf g_\mu}\bigl(\sigma|_{[-T,t_b]}\bigr)
    &\le \int_{-T}^{t_b}
         \sqrt{(1-\dot r(t))(1+\dot r(t))}\,dt \\
    &\le
    \left(\int_{-T}^{t_b}(1-\dot r(t))\,dt\right)^{1/2}
    \left(\int_{-T}^{t_b}(1+\dot r(t))\,dt\right)^{1/2} \\
    &= \sqrt{(T+t_b)^2-(1-b_\mu)^2}.
\end{align*} 
Using $T+t_b\le T-a-F_\mu(b_\mu)$ (due to \eqref{eq: tb estimate}) and adding the length of the final portion $\sigma|_{[t_b,0]}$ from \eqref{eq: final portion of sigma}, we obtain
\begin{equation*}
    L_{\mathsf g_\mu}(\sigma)
    \le
    \sqrt{\bigl(T-a-F_\mu(b_\mu)\bigr)^2-(1-b_\mu)^2}
    + T\sqrt{\mu}.
\end{equation*}
Taking the supremum over all future-directed causal curves $\sigma$
from $A_T$ to $Z_\mu$ and using $b_\mu\to 0$ and $F_\mu(b_\mu)\to 0$
as $\mu\downarrow 0$ (by explicit calculation), we conclude that
\begin{equation*}
    \limsup_{\mu\downarrow 0}\ell_{g_\mu}(A_T,Z_\mu)
    \le \sqrt{(T-a)^2-1}.
\end{equation*}
For the lower bound, we construct a future causal curve from
$A_T$ to $Z_\mu$. We set $c_\mu:=\sqrt{\mu}$. Since $\rho_\mu<4\mu$,
we have $\rho_\mu<c_\mu<1$ for all sufficiently small $\mu$. Moreover,
on the interval $c_\mu\le r\le 1$ we have
\begin{equation*}
    1-2\sqrt{\mu}=f_\mu(c_\mu)\le f_\mu(r)<1.
\end{equation*}
Recall that along the null curve $\lambda_\mu$ we have
$F_\mu(r(t))=-a-t$. Thus $\lambda_\mu$ meets the hypersurface
$r=c_\mu$ at the point
\begin{equation*}
    B_\mu:=(t_\mu,c_\mu,\omega_*),
    \qquad t_\mu:=-a-F_\mu(c_\mu).
\end{equation*}
For small enough $\mu$,
\begin{equation*}
    0\le F_\mu(c_\mu)-\sqrt{\mu}
    =2\mu\log\left(\frac{1}{2\sqrt{\mu}}-1\right) \le 2\mu \log(\mu^{-1/2})
    = -\mu\log\mu \longrightarrow 0.
\end{equation*}
Hence $t_\mu\to-a$ as $\mu\downarrow0$. Now, we connect $A_T$ to $B_\mu$ by the radial curve
\begin{equation*}
    \alpha_\mu(t):=(t,r_\mu(t),\omega_*),
    \qquad
    r_\mu(t):=1-v_\mu(t+T),
    \qquad
    v_\mu:=\frac{1-c_\mu}{T+t_\mu},
    \qquad t\in[-T,t_\mu].
\end{equation*}
Indeed, $r_\mu(-T)=1$ and $r_\mu(t_\mu)=c_\mu$. Since $T>1+a$,
we have
\begin{equation*}
    v_\mu\longrightarrow \frac{1}{T-a}<1.
\end{equation*}
Hence, for all small enough $\mu$ and all
$t\in[-T,t_\mu]$,
\begin{equation*}
    |\dot r_\mu(t)|=v_\mu<1-2\sqrt{\mu}\le f_\mu(r_\mu(t)),
\end{equation*}
so $\alpha_\mu$ is future timelike. Concatenating $\alpha_\mu$
with $\lambda_\mu|_{[t_\mu,0]}$ gives a future-directed causal curve
from $A_T$ to $Z_\mu$. Since the second portion is null, we obtain
\begin{equation*}
    \ell_{\mathsf g_\mu}(A_T,Z_\mu)\ge L_{\mathsf g_\mu}(\alpha_\mu).
\end{equation*}
Along $\alpha_\mu(t)=(t,r_\mu(t),\omega_*)$, we have
$\dot r_\mu(t)=-v_\mu$. Set $m_\mu:=1-2\sqrt{\mu}$. Since
$r_\mu(t)\ge c_\mu=\sqrt{\mu}$ for all $t\in[-T,t_\mu]$, we have
$f_\mu(r_\mu(t))\ge f_\mu(c_\mu)=m_\mu$. Hence
\begin{align*}
\mathsf g_\mu(\dot\alpha_\mu(t),\dot\alpha_\mu(t))
    = f_\mu(r_\mu(t))-\frac{v_\mu^2}{f_\mu(r_\mu(t))} 
    \ge m_\mu-\frac{v_\mu^2}{m_\mu}>0,
\end{align*}
where the last inequality holds for sufficiently small $\mu$ since
$v_\mu<m_\mu$. Integrating over $[-T,t_\mu]$ gives
\begin{align*}
    L_{\mathsf g_\mu}(\alpha_\mu)
    \ge (T+t_\mu)\sqrt{m_\mu-\frac{v_\mu^2}{m_\mu}} 
    = \sqrt{m_\mu(T+t_\mu)^2-\frac{(1-c_\mu)^2}{m_\mu}}.
\end{align*}
Since $m_\mu\to1$,
$c_\mu\to0$, and $t_\mu\to-a$ as $\mu\downarrow0$, it follows that
\begin{equation*}
    \liminf_{\mu\downarrow0}\ell_{\mathsf g_\mu}(A_T,Z_\mu)
    \ge \liminf_{\mu\downarrow0}L_{\mathsf g_\mu}(\alpha_\mu)
    \ge \sqrt{(T-a)^2-1}.
\end{equation*}
Combining this with the upper bound yields
\begin{equation}
\label{eq: lim ell(AT,Zmu)}
    \lim_{\mu\downarrow0}\ell_{\mathsf g_\mu}(A_T,Z_\mu)
    =\sqrt{(T-a)^2-1}
    \qquad\text{for every fixed }T>1+a.
\end{equation}
\noindent \textbf{STEP 3: Transferring to the long timelike segment.} Now fix $T>1+a$ and write $Z_j:=Z_{\mu_j}$ and $\ell_j:=\ell_{\mathsf g_{\mu_j}}$.
We now show that \eqref{eq: lim ell(AT,Zmu)} holds with $A_T=(-T,1,\omega_*)$ replaced by
$\gamma_j(-T)$. Fix $0<\varepsilon<T-(1+a)$. We claim that, for all
sufficiently large $j$,
\begin{equation*}
A_{T+\varepsilon}\ll_{\mathsf g_{\mu_j}}\gamma_j(-T)
    \ll_{\mathsf g_{\mu_j}}A_{T-\varepsilon}.
\end{equation*}
This is easily seen from our estimates \eqref{eq: r(t) - 1} and \eqref{eq: t - s(t)}, which in particular imply that $\gamma_j(-T)\to(-T,1,\omega_*)$. This observation, together with \eqref{eq: lim ell(AT,Zmu)}, readily gives
\begin{equation*}
    \lim_{j\to\infty}\ell_j(\gamma_j(-T),Z_j)
    =\sqrt{(T-a)^2-1}.
\end{equation*}
The analogous considerations also yield
\begin{equation*}
\lim_{j\to\infty}\ell_j(Z_j,\gamma_j(T))
    =\sqrt{(T-a)^2-1}.
\end{equation*}
Recalling that $a=1/2$ and taking $T=4$ and $T=2$, we conclude that
\begin{equation}
\label{eq: various time sep}
    \ell_j(\gamma_j(-4),Z_j)^2 \longrightarrow \frac{45}{4}, \quad
    \ell_j(Z_j,\gamma_j(4))^2  \longrightarrow \frac{45}{4},\quad
    \ell_j(\gamma_j(-2),Z_j)^2 \longrightarrow \frac{5}{4}.
\end{equation}
\noindent \textbf{STEP 4: Four-point contradiction to near-product geometry.} We are now in position to show that Theorem \ref{Theorem: main theorem} does not hold under a timelike Ricci lower bound. By way of contradiction, suppose it does hold, we apply it with $R=4$, $L=j$, and $\delta=0$.
Indeed, $(\mathcal M_{\mu_j},\mathsf g_{\mu_j})$ is globally hyperbolic and Ricci-flat,
and $\gamma_j\colon[-j,j]\to \mathcal M_{\mu_j}$ is a unit-speed maximizing
timelike geodesic. We refer to Figure \ref{fig:schwarzschild-four-points} for a schematic overview of the construction in this proof. We set
\begin{equation}\label{Schwarzschild diamond}
    P_j:=\gamma_j(-4),\qquad C_j:=\gamma_j(-2),\qquad
    Q_j:=\gamma_j(4),\qquad D_j:=J^{\mathcal M_{\mu_j}}(P_j,Q_j).
\end{equation}
By assumption, for every sufficiently large $j$
there exists a pointed metric space $(X_j,\mathsf d_j,o_j)$ such that, with
$\widehat D_j:=J^{\mathbb R\times X_j}((-4,o_j),(4,o_j))$, we have
\begin{equation*}
    d_{\mathrm{LGH}}(D_j,\widehat D_j)\le\Psi,
\end{equation*}
where the right hand side tends to
zero as $j \to \infty$. Hence, by Definition \ref{D:d_LGH}, we may choose a correspondence
$\mathcal R_j$ between $D_j$ and $\widehat D_j$ satisfying
\begin{equation*}
    \operatorname{dis}(\mathcal R_j)
    \le d_{\mathrm{LGH}}(D_j,\widehat D_j)+j^{-1}
    \le \Psi+j^{-1}\longrightarrow0.
\end{equation*}
Write $\tau_j:=(\ell_j)_+$ and
$\widehat\tau_j:=(\ell^{\mathbb R\times X_j})_+$ for the nonnegative
parts of the respective time separations. Since $\gamma_j$ is unit speed
and maximizing, we have
\begin{equation*}
    \tau_j(P_j,Q_j)=8,\qquad
    \tau_j(P_j,C_j)=2,\qquad
    \tau_j(C_j,Q_j)=6.
\end{equation*}
Moreover, \eqref{eq: various time sep} gives
\begin{equation*}
    \tau_j(P_j,Z_j)^2\longrightarrow\frac{45}{4},\qquad
    \tau_j(Z_j,Q_j)^2\longrightarrow\frac{45}{4},\qquad
    \tau_j(C_j,Z_j)^2\longrightarrow\frac{5}{4}.
\end{equation*}
In particular, $P_j\ll Z_j\ll Q_j$ for all sufficiently large $j$,
so $Z_j\in D_j$. Choose points
\begin{equation*}
    P'_j=(p_j,u_j),\qquad Q'_j=(q_j,v_j),\qquad
    C'_j=(s_j,y_j),\qquad Z'_j=(t_j,z_j)
\end{equation*}
in $\widehat D_j$ corresponding under $\mathcal R_j$ to
$P_j,Q_j,C_j,Z_j$, respectively. By the definition of distortion,
the nonnegative time separation of each ordered pair of these four
points differs from that of the corresponding pair in $D_j$ by at
most $\operatorname{dis}(\mathcal R_j)$.

We first identify $P'_j$ and $Q'_j$ in the product diamond. Since
$\widehat\tau_j(P'_j,Q'_j)\to8$, these points are chronologically
related for all sufficiently large $j$. By definition of the product time separation,
\begin{equation*}
    \widehat\tau_j(P'_j,Q'_j)\le q_j-p_j\le8,
\end{equation*}
where the last inequality follows from $p_j,q_j\in[-4,4]$.
Thus $p_j\to-4$ and $q_j\to4$. Moreover, note that
\begin{equation*}
    (h,w)\in\widehat D_j
    \quad\Longleftrightarrow\quad
    |h|+\mathsf d_j(o_j,w)\le4.
\end{equation*}
It follows that
\begin{equation*}
    \mathsf d_j(o_j,u_j)\le4-|p_j|\longrightarrow0,\qquad
    \mathsf d_j(o_j,v_j)\le4-|q_j|\longrightarrow0.
\end{equation*}
Next, we identify $C'_j$ and $Z'_j$ using their time separations from $P'_j$
and $Q'_j$. Take $W_j=C_j$ or $W_j=Z_j$, and write
$W'_j=(h_j,w_j)$ for its chosen corresponding point. In both cases,
$\tau_j(P_j,W_j)$ and $\tau_j(W_j,Q_j)$ have positive limits, so
$P'_j\ll W'_j\ll Q'_j$ for all sufficiently large $j$.
By the triangle inequality,
\begin{equation*}
    \big|\mathsf d_j(u_j,w_j)-\mathsf d_j(o_j,w_j)\big|\le \mathsf d_j(u_j,o_j)\to0,
\end{equation*}
and the same holds with $v_j$ in place of $u_j$. Squaring and the vanishing distortion imply
\begin{align}
    \tau_j(P_j,W_j)^2
    &= (h_j+4)^2-\mathsf d_j(o_j,w_j)^2+o(1), \label{eq: tjPjWj}\\
    \tau_j(W_j,Q_j)^2
    &= (4-h_j)^2-\mathsf d_j(o_j,w_j)^2+o(1). \label{eq: tjWjQj}
\end{align}
Subtracting the second equation from the first gives
\begin{equation}
\label{eq: tausquared - tausquared}
    16h_j=\tau_j(P_j,W_j)^2-\tau_j(W_j,Q_j)^2+o(1).
\end{equation}
Taking $W_j=C_j$, we obtain $16s_j\to 2^2-6^2=-32$, hence $s_j\to-2$.
Inserting this into \eqref{eq: tjPjWj} gives
\begin{equation*}
    \mathsf d_j(o_j,y_j)^2=(s_j+4)^2-2^2+o(1)\longrightarrow0.
\end{equation*}
Taking $W_j=Z_j$, the two squared time separations in \eqref{eq: tausquared - tausquared} have the same limit
$45/4$, hence $t_j\to0$. Hence,
\begin{equation*}
    \mathsf d_j(o_j,z_j)^2
    =(t_j+4)^2-\tau_j(P_j,Z_j)^2+o(1)
    \longrightarrow 16-\frac{45}{4}=\frac{19}{4}.
\end{equation*}
Finally, consider the pair $C'_j,Z'_j$. Its time-coordinate difference
satisfies $t_j-s_j\to2$. On the other hand,
\begin{equation*}
    \big|\mathsf d_j(y_j,z_j)-\mathsf d_j(o_j,z_j)\big|
    \le \mathsf d_j(y_j,o_j)\longrightarrow0,
\end{equation*}
so $\mathsf d_j(y_j,z_j)\to\sqrt{19}/2>2$. Thus
$t_j-s_j<\mathsf d_j(y_j,z_j)$ for all sufficiently large $j$ and so $\widehat\tau_j(C'_j,Z'_j)=0$ by definition of the product time separation. For those $j$,
the definition of distortion implies
\begin{equation*}
    \operatorname{dis}(\mathcal R_j)
    \ge \big|\tau_j(C_j,Z_j)-\widehat\tau_j(C'_j,Z'_j)\big|
    =\tau_j(C_j,Z_j)\longrightarrow\frac{\sqrt{5}}{2}>0.
\end{equation*}
This contradicts $\operatorname{dis}(\mathcal R_j)\to0$. Consequently, the timelike sectional curvature hypothesis in Theorem \ref{Theorem: main theorem} cannot be replaced
by a timelike Ricci curvature lower bound, even in dimension four
and with $\operatorname{Ric}\equiv0$.

\begin{figure}[htbp]
\centering
\begin{tikzpicture}[
  y=.68cm, >=stealth, font=\small,
  axis/.style={->,line width=.28pt},
  geodesic/.style={line width=1.05pt},
  timelike/.style={line width=.5pt},
  null/.style={line width=.55pt,dash pattern=on 2.2pt off 1.7pt},
  event/.style={circle,fill=black,inner sep=1.4pt},
  auxiliary/.style={circle,draw,fill=white,inner sep=1.25pt}
]

\begin{scope}[x=3.85cm]
  \draw[axis] (0,-4.32) -- (0,4.45) node[above] {$t$};
  \draw[axis] (0,0) -- (1.38,0) node[right] {$r$};
  \draw[thin,densely dotted] (.13,-4.27) -- (.13,4.27)
    node[above right,inner sep=1pt] {$r=2\mu_j$};

  \coordinate (P)  at (1,-4);
  \coordinate (C)  at (1.02,-2);
  \coordinate (Q)  at (1,4);
  \coordinate (Z)  at (.18,0);
  \coordinate (Bm) at (.32,-.64);
  \coordinate (Bp) at (.32,.64);

  \draw[timelike] (P) -- (Bm) (C) -- (Bm) (Bp) -- (Q);
  \draw[null] (Bm) .. controls (.23,-.46) and (.183,-.27) .. (Z)
                   .. controls (.183,.27) and (.23,.46) .. (Bp);
  \draw[geodesic] plot[smooth] coordinates
    {(.995,-4.35) (P) (C) (1.028,0) (1.02,2) (Q) (.995,4.35)};

  \foreach \pt/\lab in {P/P_j,C/C_j,Q/Q_j}
    \node[event,label=right:$\lab$] at (\pt) {};

  \node[event] at (Z) {};
  \draw[thin] (Z) -- (.40,.23) node[right,inner sep=1.5pt] {$Z_j$};

  \node[auxiliary,label=below left:$B_j^-$] at (Bm) {};
  \node[auxiliary,label=above left:$B_j^+$] at (Bp) {};
  \node[right,font=\scriptsize] at (.35,-.52)
    {$t_{\mu_j}\to-\frac12$};

  \node[right] at (1.02,2.45) {$\gamma_j$};
  \node[below right,font=\footnotesize] at (1.02,0) {$r\simeq1$};

  \node at (.66,-4.98) {(a) Exterior Schwarzschild};
  \node at (.66,-5.55) {$\tau_j(C_j,Z_j)\to\sqrt{5}/2$};
\end{scope}

\begin{scope}[xshift=6.85cm,x=1.04cm]
  \draw[axis] (0,-4.32) -- (0,4.45) node[above] {$t$};
  \draw[axis] (0,0) -- (4.38,0) node[right] {$\varrho$};

  \draw[thin] (0,-4) -- (4,0) -- (0,4);
  \draw[geodesic] (0,-4) -- (0,4);

  \draw[null] (0,-2) -- (2,0);
  \draw[thin] (2,.08) -- (2,-.08) node[below] {$2$};

  \foreach \h/\lab in {-4/P'_j,-2/C'_j,4/Q'_j}
    \node[event,label=left:$\lab$] at (0,\h) {};

  \node[below right,font=\scriptsize] at (0,-4) {$-4$};
  \node[below right,font=\scriptsize] at (0,-2) {$-2$};
  \node[above right,font=\scriptsize] at (0,4)  {$4$};

  \coordinate (Zp) at ({sqrt(19)/2},0);
  \node[event] at (Zp) {};
  \draw[thin] (Zp) -- (1.62,.8) node[above] {$Z'_j$};

  \node[font=\footnotesize] at (3.35,2.4)
    {$\varrho=\mathsf{d}_j(o_j,w)$};

  \node at (1.65,-4.98) {(b) Limiting product coordinates};
  \node at (1.65,-5.55)
    {$\widehat{\tau}_j(C'_j,Z'_j)=0$ for large $j$};
\end{scope}

\begin{scope}[yshift=-4.47cm,font=\footnotesize]
  \draw[geodesic] (0,0) -- (.46,0)
    node[right] {long maximizer};
  \draw[timelike] (3.75,0) -- (4.21,0)
    node[right] {timelike comparison};
  \draw[null] (8.12,0) -- (8.58,0)
    node[right] {radial null segment};
\end{scope}
\end{tikzpicture}

\caption[Four-point obstruction to almost splitting]{%
  Four-point obstruction to almost splitting. Here, $B_j^-:=B_{\mu_j}$ and $B_j^+:=\mathcal T(B_j^-)$, where $\mathcal T(t,r,\omega):=(-t,r,\omega)$.}
\label{fig:schwarzschild-four-points}
\end{figure}

\begin{Remark}[Computing the LGH-limit]
\label{remark: lgh limit}
The Schwarzschild diamonds $D_j$ from \eqref{Schwarzschild diamond} converge in the bounded Lorentzian Gromov--Hausdorff sense of Minguzzi--Suhr \cite{MinguzziSuhr22}, to the following gluing:
\[
D_\infty
=
\underbrace{\{(t,x)\in\mathbb R^{1+3}:|t|+|x-x_*|\le4\}}_{D_{\mathrm M}}
\;\cup\;
\underbrace{\{(t,a):a\ge0,\ |t|+a\le3\}}_{H},
\]
where $x_* \equiv \omega_* \in S^2 \subseteq \R^3$ and $(t,0)_H$ is identified with $(t,0)_{\mathrm M}$. Thus the gluing takes place along the one-dimensional residual of the horizon left behind by the vanishing Schwarzschild mass, and the glued triangle parameterizes possible coordinate time delays of the null curves that manage to reach the rotational symmetry axis of the diamond from sufficiently close to the tiny horizons. They must travel asymptotically close to the $x_*$ direction to do so. Thus the time separation $\tau=\ell_+$ on the ambient gluing
\[
X=\mathbb R^{1+3}
\cup_{\mathbb R\times\{0\}}
\bigl(\mathbb R\times[0,\infty)\bigr)
\]
is given with $\Phi(v,R)=\sqrt{((v_+)^2-R^2)_+}$ by
\[
\begin{aligned}
\tau_\infty((s,x)_{\mathrm M},(t,y)_{\mathrm M})
  &=\Phi(t-s,|x-y|),\\
\tau_\infty((s,x)_{\mathrm M},(t,a)_H)
  &=\Phi(t-s-a,|x|),\\
\tau_\infty((s,a)_H,(t,y)_{\mathrm M})
  &=\Phi(t-s-a,|y|),\\
\tau_\infty((s,a)_H,(t,b)_H)
  &=(t-s-a-b)_+.
\end{aligned}
\]
The convergence can be realized by explicit correspondences. 
Define $\mathcal F_j:\mathcal M_{\mu_j} \to X$ using the tortoise coordinate \eqref{tortoise} by
\[
\mathcal F_j(t,r,\omega)=
\begin{cases}
\bigl(t,(r-\sqrt{\mu_j})\omega\bigr)_{\mathrm M},
  &r\ge \sqrt{\mu_j},\\[1mm]
\bigl(t,F_{\mu_j} (\sqrt{\mu_j})-F_{\mu_j}(r)\bigr)_H,
  &2\mu_j<r<\sqrt{\mu_j}.
\end{cases}
\]
Equip the ambient gluing $X$ with the
auxiliary metric $\mathsf d_\star$ defined by
\[
\begin{aligned}
\mathsf d_\star((s,x)_{\mathrm M},(t,y)_{\mathrm M})
  &=|t-s|+|x-y|,\\
\mathsf d_\star((s,a)_H,(t,b)_H)
  &=|t-s|+|a-b|,\\
\mathsf d_\star((s,x)_{\mathrm M},(t,a)_H)
  &=|t-s|+|x|+a,
\end{aligned}
\]
and by symmetry for the reversed mixed pair.
For a sufficiently large fixed constant $C$, set
\[
\mathcal R_j
:=
\bigl\{(p,z)\in D_j\times D_\infty:
\mathsf d_\star(\mathcal F_j(p),z)\le Cj^{-2}\bigr\}.
\]
One can check that
\[
\begin{aligned}
\mathsf d_{\text{Hausdorff}}^{\mathsf d_\star}\bigl(\mathcal F_j(D_j),D_\infty\bigr)
  &=O(j^{-2}),\\
\sup_{p,q\in D_j}
\bigl|\tau_j(p,q)
-\tau_\infty(\mathcal F_j(p),\mathcal F_j(q))\bigr|
  &=O(j^{-3/2}).
\end{aligned}
\]
The first estimate ensures that $\mathcal R_j$ projects
surjectively onto both factors. Since $\tau_\infty$ is
$1/2$-H\"older with respect to $\mathsf d_\star$ on bounded sets,
the second gives
\[
\operatorname{dis}(\mathcal R_j)
\le O(j^{-3/2})+O(j^{-1})
\longrightarrow0.
\]

The resulting limit is not isometric to a diamond in a
Lorentzian product, even with a metric spatial factor.
Indeed, $z=(0,1)_H$ lies chronologically between the two tips,
but its open neighborhood
\[
U:=\{(t,a)_H:|t|<1/4,\ |a-1|<1/4\}
\]
satisfies $\tau_\infty|_{U\times U}=0$.
However, in a Lorentzian product, every chronologically
interior point of a diamond admits arbitrarily small
vertical timelike segments contained in that diamond,
with
\[
\tau_{\mathrm{prod}}((t,x),(t+\delta,x))=\delta>0.
\]
Since isometries of bounded Lorentzian metric spaces are
homeomorphisms \cite[Thm.\ 1.17]{MinguzziSuhr22},
this is incompatible with the existence of $U$.
\end{Remark}

\begin{Remark}[Unmeasured vs.\ measured limits]
\label{remark: unmeasured measured}
\begin{enumerate}
\item[]
\item As $\mu_j\downarrow0$, using the arguments presented in Remark \ref{remark: lgh limit}, one can show that the Schwarzschild exteriors $(\mathcal M_{\mu_j},\mathsf g_{\mu_j})$ converge in the
(global) sequenced Lorentzian Gromov--Hausdorff sense of Bykov--Minguzzi--Suhr \cite{BykovMinguzziSuhr}
to the ambient space $X$ of Remark \ref{remark: lgh limit}, with generating points chosen
at integer times along the limit vertical line at $x_*$. Its central
diamond of radius $4$ is precisely the gluing $D_\infty$ described there.

\item As a consequence of (i), an unmeasured Lorentzian Gromov--Hausdorff limit of globally
hyperbolic Ricci-flat spacetimes can contain a complete timelike line
and fail to split. Indeed, the glued Minkowski component of $X$ contains
complete timelike lines, whereas it is readily checked that no timelike line can pass through
a point $(t,a)_H$ with $a > 0$ in the attached half-plane.

\item Upon normalization of the volume, one can show that the diamonds $D_j$ converge
to the Minkowski diamond $D_{\mathrm M}$ in the Lorentz--Gromov box
distance of Braun and S\"amann~\cite{braunsamann}, hence also in their $p$-distortion
distance $\mathsf{L}\boldsymbol{\Delta}_p$ for $p = 0$ and for every $1 \leq p < +\infty$ \cite[Thm.\ 4.26, Cor.\ A.13]{braunsamann} (but not $p = +\infty$, cf.\ \cite[Prop.\ A.14]{braunsamann}). Thus, while an unmeasured direct Lorentzian analogue of Cheeger--Colding is false in Lorentzian signature, it is still reasonable to conjecture that an
analogue of Theorem \ref{Theorem: main theorem} holds in a distance metrizing measured LGH-convergence under an asymptotically vanishing timelike
Ricci curvature lower bound.
\end{enumerate}
\end{Remark}

\section*{Data availability statement}

There is no data associated with this manuscript.

\section*{Conflict of interest statement}

The authors confirm that no conflict of interest exists for this work.

\section*{Use of AI declaration}

Apart from Appendix \ref{section: counterexample},
all of the main results, arguments and techniques presented herein were developed solely by the authors.  ChatGPT was initially used only for mechanical calculations, such as computing coefficients in a Taylor series.  After a complete draft had been produced,  during the editing and revising process discussions with ChatGPT were used for checking the arguments and refining the writing. Following the October 2026 posting of this manuscript to the arXiv, discussions with ChatGPT were used to obtain the counterexample given in Appendix \ref{section: counterexample}. The authors take full responsibility for the final contents.

\section*{Acknowledgments}

The authors are grateful for the hospitality of the Department of Mathematics of the University of Vienna during the course of the 2nd Workshop on ``A new geometry for Einstein's theory of relativity and beyond'' and the workshop on ``Curvature, Optimal Transport and Gravity", and for the hospitality of the Bernoulli Centre at EPFL during the workshop ``Optimal Transport and Metric Geometry across Structures".

NG and MZ were supported by the Italian Ministry of University and Research (MUR) under the 'Fondo Italiano per la Scienza' (FIS 3) program, PI: prof. Nicola Gigli, Project Title: Modern challEnges in Geometric Analysis - MEGA, CUP: G53C25000920001.

MG acknowledges support from the Deutsche Forschungsgemeinschaft (DFG, German Research Foundation) under Germany’s Excellence Strategy – EXC 2121 "Quantum Universe" – 390833306.

RJM's research is supported in part by the Canada Research Chairs program CRC--2020--00289, Natural Sciences and Engineering Research Council of Canada Discovery Grant RGPIN-2020--04162 and 2026--04906, and Toronto's Fields Institute for the Mathematical Sciences.

EW's research is supported by Natural Sciences and Engineering Research Council of Canada Discovery Grant RGPIN--2022--03440.

MZ is grateful for the support of the Fields Institute for Research in Mathematical Sciences (Toronto, Canada) during the Thematic Program on Shocks and Singularities: Nonlinear evolution equations in physical and life sciences.

This research was funded in part by the Austrian Science Fund (FWF) [Grants DOI\\10.55776/EFP6, 10.55776/STA32 and 10.55776/J4913]. For open access purposes, the authors have applied a CC BY public copyright license to any author accepted manuscript version arising from this submission.

\addcontentsline{toc}{section}{References}
\bibliography{Bibliography} 
\bibliographystyle{abbrv}

\end{document}